\documentclass[12pt]{article}

\usepackage[T2A]{fontenc}
\usepackage[utf8]{inputenc}
\usepackage[russian,ukrainian,english]{babel}

\usepackage{amssymb}
\usepackage{graphics}
\usepackage{mathtools} 

\newcommand{\qed}{$\;\;\;\Box$}
\newenvironment{proof}{\par\smallbreak{\sl Proof.~}}
{\unskip\nobreak\hfill \qed \par\medbreak}

\newcounter{claim}
\renewcommand{\theclaim}{\arabic{claim}}
\newenvironment{claim}{\refstepcounter{claim}%
	\par\medskip\par\noindent{\bf Claim~\theclaim.}\rm}%
{\par\medskip\par}

\newenvironment{subproof}{\par\noindent{\bf Proof of Claim.}}%
{\qed\par\smallbreak}
\newcommand{\hide}[1]{}

\newcommand{\beq}{\begin{equation}}
	\newcommand{\ee}{\end{equation}}

\newtheorem{thm}{Theorem}[section]
\newtheorem{lemma}[thm]{Lemma}

\newtheorem{defn}[thm]{Definition}

\renewcommand{\Re}{\mathop{\mathrm{Re}}\nolimits}

\title{On the bounded smooth solutions to exponentially stable linear nonautonomous  hyperbolic systems
}

\newcounter{thesame}
\author{
	Irina Kmit
	\thanks{Institute of Mathematics, Humboldt University of Berlin. On leave from the
		Institute for Applied Problems of Mechanics and Mathematics,
		Ukrainian National Academy of Sciences, Lviv. {\small   E-mail:
			{\tt irina.kmit@hu-berlin.de}}}
	\ \ \ Viktor Tkachenko \thanks{Institute of Mathematics,
		National Academy of Sciences of	Ukraine, Kyiv.
		{\small   E-mail:
			{\tt vitk@imath.kiev.ua}}
}}

\date{}

\begin{document}
	
	\maketitle






\begin{abstract}
We investigate global bounded solutions of higher regularity  to  boundary value problems
for a general linear nonautonomous first order 1D hyperbolic system in a strip.
We establish the existence of such solutions under the assumption of exponential stability and certain dissipativity (or non-resonant) conditions. Our results demonstrate the existence and uniqueness of $C^1$- and $C^2$-bounded solutions, in particular,   periodic and almost periodic solutions. The number of imposed dissipativity conditions is related to the regularity of solutions. This connection arises due to the nonautonomous nature of the hyperbolic system under consideration, 
in contrast to  autonomous settings. 

In the general case of global bounded smooth solutions we assume  exponential stability in $H^1$. In the  case of periodic solutions, the weaker assumption of $L^2$-exponential stability proves to be sufficient. 
\end{abstract}


\section{Introduction}
\renewcommand{\theequation}{{\thesection}.\arabic{equation}}
\setcounter{equation}{0}

\subsection{Problem setting and our results} 
\renewcommand{\theequation}{{\thesection}.\arabic{equation}}
\setcounter{equation}{0}

We consider general
linear  first order hyperbolic systems   of the  type
\begin{equation}\label{eq:1}
\partial_t u + a(x,t)\partial_x u + b(x,t)u  = f(x,t),\quad x\in (0,1),
\end{equation}
where  $n\ge 2$ is a fixed integer, $u=(u_1,\ldots,u_n)$ and $f = (f_1,\dots,f_n)$ are  vectors of real-valued functions,
$a=diag(a_1,\dots,a_n)$ and $b = (b_{jk})$ are $n\times n$-matrices with real-valued entries.

Set
$$
\Pi = \{(x,t)\in\mathbb R^2\,:\,0\le x\le1\}.
$$
Suppose that there exist an integer $m$ in the range  $0\le m\le n$ and    a positive real $a_0$ such that
\begin{equation}\label{eq:h1}
	\begin{array}{ll}
		\inf\left\{a_j(x,t)\,:\,(x,t)\in\Pi,\ 1\le j\le m\right\}\ge a_0,\\ [1mm]
		\sup\left\{a_j(x,t)\,:\,(x,t)\in\Pi,\ m+1\le j\le n\right\}\le -a_0,\\ [1mm]
		\inf\left\{|a_j(x,t)-a_k(x,t)|\,:\,\,(x,t)\in\Pi, \ 1\le j\ne k\le n\right\}\ge a_0.
	\end{array}
\end{equation}                                                                                                                                                                                            
We  supplement the system (\ref{eq:1})  with the reflection boundary conditions accordingly to 
the sign of $a_j$, namely
\begin{equation}\label{eq:2}
\begin{array}{ll}
u_{j}(0,t)= R_ju(\cdot,t), \quad 1\le j\le m,\\ [1mm]
 u_{j}(1,t)= R_ju(\cdot,t), \quad m< j\le n,
\end{array}
\end{equation}
where the operators $R_j$ are defined either on the space $H^1((0,1); \mathbb R^n)$ or  on the space $C([0,1];\mathbb R^n)$ and  are given by
\begin{eqnarray}
 \displaystyle R_jv=\sum\limits_{k=m+1}^nr_{jk} v_k(0)+\sum\limits_{k=1}^mr_{jk} v_k(1),\quad  j\le n,
\label{eq:R} 
\end{eqnarray}
and $r_{jk}$ are real constants.  
Here and in what follows,  we use a standard notation $H^1\left((0,1);\mathbb R^n\right)$ for the Sobolev space of 
functions $v\in L^2\left((0,1);\mathbb R^n\right)$ whose first order distributional derivatives
 belong to~$L^2\left((0,1);\mathbb R^n\right)$.
Write $R=(R_1,\dots,R_n)$.

Note that the problem  (\ref{eq:1}), (\ref{eq:2})
 is especially interesting as it covers various physical phenomena
\cite{BCN,DDTK1,DDTK2,PW}. 

The objective of this paper  is to establish sufficient conditions 
for the existence and
 uniqueness of global bounded classical ($C^1$- and $C^2$-continuously differentiable) solutions to the problem  (\ref{eq:1}), (\ref{eq:2}) in the strip $\Pi$.
 In the case that the coefficients $a, b,$ and $f$ exhibit almost periodic (or $T$-periodic) behavior, our analysis will reveal that the resulting bounded solution inherits almost periodicity (or $T$-periodicity).

 We use the same notation $\|\cdot\|$ for norms of vectors and matrices of different sizes. Thus, if $A=(A_1,\dots,A_k)$ is
 a $k$-vector, then $\|A\|=\max_{j\le k}|A_j|$. If $A=(A_{ij})$ is
 a $k\times k$-matrix, then $\|A\|=\max_{i,j\le k}|A_{ij}|$.

Let $\Omega$ be a subdomain  of $\Pi$ or  $\mathbb R$ and let $X$ be a Banach space.
Then write $BC(\Omega;X)$ to denote  the Banach space of
all bounded and continuous maps
$u:\Omega \to X$,
with norm
$$
\|u\|_{BC(\Omega;X)}=\sup\left\{\|u(z)\|_X\,:\,z\in\Omega\right\}.
$$
 We also use the space $BC^k_t(\Omega; X)$ of functions 
$u \!\in\! BC(\Omega;X)$ such that $\partial_t u, \dots, \partial_t^k u \!\in\! BC(\Omega;X),$ with norm
$$\|u\|_{BC_t^k(\Omega;X)}= \sum_{j=0}^k  \|\partial_t^j u\|_{BC(\Omega;X)}.$$
Similarly, for given $k\ge 1$, one can  introduce the space $BC^k(\Omega; X)$ of bounded and $k$-times 
continuously differentiable in $\Omega$ functions.

As usual, by ${\mathcal L}(X,Y)$ we will denote the space of linear
bounded operators from $X$ into~$Y$, and write ${\mathcal L}(X)$ for ${\mathcal L}(X,X)$.

 Let $s\in\mathbb R$ be arbitrary fixed. Set $\Pi_s=\{(x,t)\in\Pi: t\ge s\}$.
In the domain $\Pi_s$ consider the linear
 homogeneous system
\begin{equation}\label{eq:1u}
	\partial_tu  + a(x,t)\partial_x u + b(x,t) u = 0, \quad x\in(0,1),
\end{equation}
and supplement it with the  boundary conditions (\ref{eq:2})
and the initial conditions
\begin{eqnarray}
u(x,s) = \varphi(x), \quad x\in[0,1].
\label{eq:in}
\end{eqnarray}

 We say that a function $\varphi\in C([0,1];\mathbb R^n)$ satisfies the zero order compatibility conditions if 
 it fulfills the equations
\begin{equation}\label{zero}
 \begin{array}{ll}
 \displaystyle\varphi_j(0)	= \sum\limits_{k=m+1}^nr_{jk} \varphi_k(0)+\sum\limits_{k=1}^mr_{jk} \varphi_k(1), \quad 1\le j\le m,\\ [1mm]
 \displaystyle \varphi_j(1)	= \sum\limits_{k=m+1}^nr_{jk} \varphi_k(0)+\sum\limits_{k=1}^mr_{jk} \varphi_k(1), \quad m< j\le n.
 \end{array}
\end{equation}
 
Our starting point is 
  that, if the coefficients of (\ref{eq:1u}) are sufficiently smooth, then the problem 
  (\ref{eq:1u}), (\ref{eq:2}), (\ref{eq:in}) generates the evolution family 
on $L^2\left((0,1);\mathbb R^n\right)$ (see Theorem~\ref{evol0} below).  To formulate this result, note that, if $\varphi\in C_0^1([0,1];\mathbb R^n)$, then it satisfies the zero order compatibility conditions  and
 the problem (\ref{eq:1u}), (\ref{eq:2}), (\ref{eq:in})  has a unique classical solution accordingly to
\cite[Theorem 2.1]{ijdsde}. 

 \begin{defn}\label{L2}\rm
Let $\varphi\in L^2((0,1);\mathbb R^n)$. A function $u\in C\left([s,\infty); L^2\left((0,1);\mathbb R^n\right)\right)$ is called
an {\it $L^2$-generalized  solution} to the problem (\ref{eq:1}), (\ref{eq:2}), (\ref{eq:in})
 if, for any sequence $(\varphi^l)$ with
$\varphi^l\in C_0^1([0,1];\mathbb R^n)$ converging to the function
 $\varphi$ in $L^2\left((0,1);\mathbb R^n\right)$,
 the sequence $(u^l)$ with $u^l\in C^1\left([0,1]\times[s,\infty);\mathbb R^n\right)$
of  classical solutions to
(\ref{eq:1}), (\ref{eq:2}), (\ref{eq:in}) with
 $\varphi$ replaced by $\varphi^l$   fulfills the convergence
$$
\|u(\cdot,\theta)-u^l(\cdot,\theta)\|_{L^2\left((0,1);\mathbb R^n\right)} \to 0 \mbox{ as } l\to\infty,
$$
uniformly in $\theta$ varying in the range $s\le \theta\le t$, for every $t>s$.
\end{defn}

\begin{thm}\cite[Theorem 2.3]{KL1}\label{evol0}
Suppose that  the coefficients $a_j$ and $ b_{jk}$ of the system~(\ref{eq:1u}) belong to
$BC^1(\Pi)$ and
the conditions (\ref{eq:h1}) are fulfilled. Then, for given  $\varphi \in L^2\left((0,1);\mathbb R^n\right)$, there exists a unique
$L^2$-generalized  solution $u:(0,1)\times\mathbb R\to \mathbb R^n$ to the problem
(\ref{eq:1u}), (\ref{eq:2}), (\ref{eq:in}). 
Moreover, the map
$$
\varphi\mapsto U(t,s)\mathbb\varphi:=u(\cdot,t)
$$
from $L^2\left((0,1);\mathbb R^n\right)$ to itself
defines  a strongly continuous, exponentially bounded evolution family  $U(t,s)\in {\mathcal L}\left(L^2\left((0,1);\mathbb R^n\right)\right)$, which means that
\vspace{1mm}

(a)  $U(t,t)=I$ and $U(t,s)=U(t,r)U(r,s)$  for all $t\ge r\ge s,$
\vspace{1mm}

 (b) the map $(t,s)\in \mathbb R^2\mapsto U(t,s)\varphi \in L^2\left((0,1);\mathbb R^n\right)$ is continuous for all $t\ge s$ and every
$\varphi\in L^2\left((0,1);\mathbb R^n\right)$,
\vspace{1mm}

(c) there exist $K \ge 1$ and $\nu \in \mathbb R$ such that
$$
\|U(t,s)\|_{{\mathcal L}\left(L^2\left((0,1);\mathbb R^n\right)\right)}\le Ke^{\nu(t-s)} \mbox{ for all } t\ge s.
$$
\end{thm}

Before stating our main results, we describe regularity properties of the evolution family.
 
\begin{thm}\label{evol}
	Suppose that  the coefficients $a_j$ and $ b_{jk}$ of the system (\ref{eq:1u}) belong to
	$BC^1(\Pi)$ and
	the conditions (\ref{eq:h1}) are fulfilled. Then, for given   
	$\varphi\in H^1\left((0,1);\mathbb R^n\right)$ satisfying the zero order compatibility 
	conditions (\ref{zero}), 
		 the map $(t,s)\in \mathbb R^2\mapsto U(t,s)\varphi \in 
		 L^2\left((0,1);\mathbb R^n\right)$ is continuously differentiable  for all $t\ge s$.
		Furthermore, 
		\begin{equation}\label{cap}
		U(\cdot,s)\varphi \in C([s_1, s_2]; H^1((0,1); \mathbb R^n))\cap
		C^1([s_1, s_2]; L^2((0,1); \mathbb R^n))
		\end{equation}
		for every $s_2 > s_1.$ 
\end{thm}

We now recall the concept of a  characteristic curve.
 For given $j\le n$, $x \in [0,1]$, and $t \in \mathbb R$, the $j$-th characteristic of (\ref{eq:1u})
passing through the point $(x,t)\in\Pi$ is defined
as the solution 
$$
\xi\in [0,1] \mapsto \omega_j(\xi)=\omega_j(\xi,x,t)\in \mathbb R
$$ 
of the initial value problem
$$\partial_\xi\omega_j(\xi, x,t)=\frac{1}{a_j(\xi,\omega_j(\xi,x,t))},\;\;
\omega_j(x,x,t)=t.$$
Due to the assumption (\ref{eq:h1}), the characteristic curve $\tau=\omega_j(\xi,x,t)$ reaches the
boundary of $\Pi$ in two points with distinct ordinates. Let $x_j$
denote the abscissa of that point whose ordinate is smaller.
Notice that  the value of  $x_j$
does not depend on $x,t$, and  it holds
$$x_j=\left\{
 \begin{array}{rl}
 0 &\mbox{if}\ 1\le j\le m\\
 1 &\mbox{if}\ m<j\le n.
\end{array}
\right.$$

For $i=0,1,2$ and $j\le n$, we will use the notation
\begin{eqnarray*}
c_j^i(\xi,x,t)=\exp \int_x^\xi
\left[\frac{b_{jj}}{a_{j}} - i\frac{\partial_t a_{j}}{a_{j}^2} \right](\eta,\omega_j(\eta))\,d\eta,\quad
d_j^i(\xi,x,t)=\frac{c_j^i(\xi,x,t)}{a_j(\xi,\omega_j(\xi))}.
\end{eqnarray*}
We write $c_j(\xi,x,t) $ and $d_j(\xi,x,t)$ for  $c_j^0(\xi,x,t)$ and  $d_j^0(\xi,x,t)$, respectively. Introduce operators $G_0, G_1, G_2 \in {\mathcal L}(BC(\mathbb R; \mathbb R^n))$ by
\begin{eqnarray} \label{eq:G}
 [G_i\psi]_j(t) = c_j^i(x_j, 1 - x_j, t)
 \sum\limits_{k=1}^nr_{jk}\psi_k(\omega_j(x_j, 1 - x_j, t)),\quad 
j \le n, \  \ i = 0,1,2.
\end{eqnarray}

Accordingly to \cite[p. 55]{Cord}, a continuous function $w(x,t)$ defined on $\Pi$
is   {\it Bohr almost periodic in $t$ uniformly in $x$ } 
if for
every $\mu > 0$ there exists a relatively dense set of $\mu$-almost
periods of~$w$, i.e., for every $\mu > 0$ there exists a positive number $l$
such that every interval of length $l$ on $\mathbb R$ contains a number $\tau$ such that
$$
| w(x,t + \tau) - w(x,t)| < \mu \ \ \mbox{ for  all }  (x,t) \in \Pi.
$$
Denote by $AP(\mathbb R; \mathbb R^n)$  the space of  Bohr  almost periodic vector-functions
$u:\mathbb R\to\mathbb R^n$. 
Analogously,  denote by $AP(\Pi; \mathbb R^n)$  the space of 
Bohr  almost periodic functions in $t$ uniformly in $x\in[0,1]$.
 If $n=1$, then we will simply write 
 $AP(\mathbb R)$ and  $AP(\Pi)$.

\begin{defn}\rm
		A function $u\in BC^1(\Pi;\mathbb R^n)$ satisfying (\ref{eq:1}), (\ref{eq:2}) pointwise is called
		a {\it bounded classical solution} to
		(\ref{eq:1}), (\ref{eq:2}) in $\Pi$. 
\end{defn}

We are now prepared to state our main results.
Our primary assumptions are the exponential stability of the system and certain dissipativity conditions. 

\begin{thm} \label{lin-smooth}
		${\bf (i)}$ Let
			$a_j, b_{jk}, f_j \in BC^1(\Pi)$ and $\partial^2_{xt}a_j, \partial^2_{xt}f_{j}\in BC(\Pi)$
		for all $j,k \le n$. Assume that  the evolution family $U(t,s)$ 
	generated by the problem (\ref{eq:1u}), (\ref{eq:2})  
	satisfy the exponential stability estimates 
		\begin{eqnarray}
		  \|U(t,s)\varphi\|_{L^2((0,1); \mathbb R^n)} \le M e^{-\alpha(t-s)}\|\varphi\|_{L^2((0,1); \mathbb R^n)},\quad t \ge s,\label{est1} \\	
	  \|U(t,s)\psi\|_{H^1((0,1); \mathbb R^n)} \le M e^{-\alpha(t-s)}\|\psi\|_{H^1((0,1); \mathbb R^n)}, \quad t \ge s, 
\label{est2}\end{eqnarray}
`
Moreover, assume that  the dissipativity conditions
	\begin{eqnarray} \label{G_i}
	\|G_i\|_{{\mathcal L}(BC(\mathbb R; \mathbb R^n))} < 1
	\end{eqnarray}
	are fulfilled  for $i=0$ and $i=1$. Then
	the system  (\ref{eq:1}), (\ref{eq:2}) has a unique bounded classical solution $u^*$. 
	Moreover, the a priori estimate 
	\begin{equation} \label{est21}
	\|u^*\|_{BC^1(\Pi;\mathbb R^n)} \le L_{1}\left(\| f\|_{BC^1(\Pi;\mathbb R^n)} + \| \partial^2_{xt}f\|_{BC(\Pi;\mathbb R^n)} \right)
	\end{equation}
	is fulfilled for a constant $L_1$ not depending on $f.$
	
	${\bf (ii)}$ Assume, additionally, that $a_j, b_{jk},f_j\in BC_t^2(\Pi)$ for all $j,k \le n$, and the dissipativity condition
	 (\ref{G_i}) is satisfied for $i=2$. Then
	$u^* \in BC^2(\Pi;\mathbb R^n)$.
	Moreover, the a priori estimate 
	\begin{equation} \label{est22}
	\|u^*\|_{BC^2(\Pi;\mathbb R^n)}   \le L_2
	\left(\| f\|_{BC^1(\Pi;\mathbb R^n)}+ \| \partial^2_{xt}f\|_{BC(\Pi;\mathbb R^n)} +\|\partial_t^2 f\|_{BC\left(\Pi;\mathbb R^n\right)}\right)
	\end{equation}
	is fulfilled for a constant $L_2$ not depending on $f.$
	
	${\bf (iii)}$ If, additionally to the conditions of Claim $(ii)$,    it holds that $a_j, b_{jk}, f_j \in AP(\Pi)$ for all $j,k \le n$, 
  	then the bounded classical solution $u^*$ belongs to~$AP(\Pi;\mathbb R^n)$.
\end{thm}

In the case of  time-periodic solutions, Theorem \ref{lin-smooth} can be improved 
by weakening the exponential stability assumption as well as by strengthening  the  a priori estimates. 
While in the general setting addressing  bounded solutions
 we assume exponential stability in $H^1$ (conditions (\ref{est1}) and (\ref{est2})), in the particular case of periodic solutions we assume 
 the $L^2$-exponential stability only (condition (\ref{est1})). Moreover, the a priori estimates
(\ref{est211}) and (\ref{L6}) for periodic solutions 
show more optimal regularities  between the   right hand sides  and the
 solutions, if  compared to  the a priori estimates (\ref{est21}) and (\ref{est22}) for bounded solutions.

Let $T>0$ be arbitrary fixed. Denote by $C_{per}(\Omega;\mathbb R^n)$  the vector space of all continuous maps $u : \Omega\to\mathbb R^n$
which are $T$-periodic in $t$, endowed with the usual supremum norm.
We also  introduce  spaces 
$C_{per}^1(\Omega;\mathbb R^n)=C^1(\Omega;\mathbb R^n)\cap C_{per}(\Omega;\mathbb R^n)$
and $C_{per}^2(\Omega;\mathbb R^n)=C^2(\Omega;\mathbb R^n)\cap C_{per}(\Omega;\mathbb R^n).
$

\begin{thm} \label{periodic}
	${\bf (i)}$ 
			Assume that $a_j, b_{jk}\in C^1_{per}(\Pi)$ and
	 $ f_j\in C_{per}(\Pi)\cap C^1_t(\Pi)$ for all $j,k \le n$. Moreover, suppose 
	that the evolution family $U(t,s)$  
	generated by the problem (\ref{eq:1u}), (\ref{eq:2})
	satisfies the $L^2$-exponential stability  estimate  (\ref{est1}).
	If   the dissipativity conditions (\ref{G_i})
	are fulfilled  for $i=0$ and $i=1$, then
	the system  (\ref{eq:1}), (\ref{eq:2}) has a unique  classical solution $u^*\in C_{per}^1(\Pi;\mathbb R^n)$.
	Moreover, the a priori estimate 
	\begin{equation} \label{est211}
		\|u^*\|_{BC^1(\Pi;\mathbb R^n)} \le L_3\| f \|_{BC_t^1(\Pi;\mathbb R^n)}
	\end{equation}
	is fulfilled for a constant $L_3$ not depending on $f.$
	
	${\bf (ii)}$ Assume, additionally, that $a_j, b_{jk}\in C_t^2(\Pi)$ and $ f_j\in C^1(\Pi)\cap C^2_t(\Pi)$  for all $j,k \le n$ and that the 
	dissipativity condition
	(\ref{G_i}) is true for $i=2$. Then
$u^* \in C^2_{per}(\Pi, \mathbb R^n)$.
	Moreover, the a priori estimate 
	\begin{equation} \label{L6}
	\|u^*\|_{BC^2(\Pi;\mathbb R^n)}   \le L_4
\left(\| f\|_{BC^1(\Pi;\mathbb R^n)}+\|\partial_t^2 f\|_{BC\left(\Pi;\mathbb R^n\right)}\right)		
		\end{equation}
		is fulfilled for a constant $L_4$ not depending on $f.$ 
\end{thm}

\subsection{Relation to previous work}
 
The current paper extends the program
suggested in~\cite{KRT}
for constructing small global bounded smooth solutions to 1D boundary value problems associated with quasilinear first order hyperbolic systems. 
The extension  targets the linearized problems
 and goes beyond the scope of  \cite{KRT}, where the reflection boundary conditions were only
 of the smoothing type. 

Our approach  involves the approximation of  $C^2$-solutions to quasilinear problems by  $C^2$-solutions to linear problems. 
Note that the approximation  idea has been successively applied in the literature, see e.g., \cite{Li, Li1, Qu0, Qu1}. 

In the present paper we focus on establishing conditions that ensure the existence of $C^2$-solutions for general linear boundary value problems of the type (\ref{eq:1}), (\ref{eq:2}). 
As an important assumption, we require the evolution operator of the homogeneous problem to be exponentially stable 
(in \cite{KRT} it is supposed to be exponentially dichotomous).
In other words,   we suppose that  the  linear problem demonstrates a regular behavior, which can be used  to  approximate  smooth solutions (without singularities  as blow-ups and shocks) of the nonlinear problem.

Note that such regular behavior of the evolution operator is known to be equivalent to the existence of a unique mild solution to the nonhomogeneous problem \cite[Theorem 1.1]{Latushkin}, for any given right hand side.  
Having at hands the mild solution, the following two important results (of general nature)  for linear systems have to be obtained: higher regularity of mild solutions and stability under perturbations in coefficients.
While the regularity and perturbation results in \cite{KRT} essentially use a smoothing property of the evolution 
operator (see \cite{kmit,KL1} for the smoothing results), in the current paper we have to overcome
 the well-known phenomenon of ``loss of smoothness''. To handle the loss of smoothness, we  impose certain dissipativity conditions.
Unlike \cite{KRT}, where the evolution operator was assumed to have exponential dichotomy (in particular, exponential stability) in $L^2$, our current assumption is that the system is exponentially stable in the higher-order regularity space $H^1$. This  seems to be essential, since the evolution operator  is not necessarily eventually differentiable.
Specifically, for $C^k$-solvability ($k\ge 1$) of the problem (\ref{eq:1}), (\ref{eq:2}), we assume that the system is $H^1$-exponentially stable, and further, that $k+1$ dissipativity conditions are satisfied.
Remark that the number of dissipativity conditions is related to the regularity of solutions. This connection arises due to the nonautonomous nature of the hyperbolic system under consideration, 
in contrast to  autonomous settings. 
Note that
 in the case of smoothing boundary conditions, the evolution operator is automatically dissipative and, therefore, 
 no dissipativity conditions are needed to be imposed. 
 
We also remark that the question of exponential stability in nonlinear problems is typically related to higher-order regularity spaces such as $H^1$, $H^2$, and $C^1$, as discussed, e.g., in  \cite{BC2, DGL, Ha, HSh}.

 Problems of the type (\ref{eq:1}), (\ref{eq:2}) have been investigated, e.g.,   in 
\cite{KM,KR,jee,KRT,Li, Qu0, Qu1, Zhang}. Our present results, therefore, extend 
the class of linear(ized) problems  for which the existence of bounded 
smooth solutions in a strip  can be proved under natural assumptions.
This extention is achieved due to the  generality of the system (\ref{eq:1}), (\ref{eq:2})
as well as due to the conditions considered for (\ref{eq:1}) and (\ref{eq:2}).

\subsection{On the exponential stability and dissipativity conditions}
\subsubsection{Examples}
Since our main	
solvability conditions (\ref{est1}), (\ref{est2}), and (\ref{G_i}) stated in Theorem \ref{lin-smooth}
(or the conditions (\ref{est1}) and (\ref{G_i}) stated in Theorem \ref{periodic}) are formulated quite generally, we here give two examples
showing that these conditions, first, are consistent
 and, second, allow large coefficients in the differential equations and/or in the
boundary conditions. The  examples are given for  time-periodic problems but can also be easily extended to
bounded solutions. The only difference is that in the latter case we have to check the exponential stability in $H^1$ instead of
 in $L^2$ as in the former case.

{\it Example 1.} 
Consider a time-periodic problem for the hyperbolic system
\begin{equation} \label{eq:1ex}
	\begin{array}{ll}
		\partial_t u_1 + \partial_x u_1  = 0,\qquad
		\partial_t u_2 -\partial_x u_2  +\alpha u_2 -u_1 =0 	
	\end{array}
\end{equation}
with the boundary conditions
\begin{equation}  \label{eq:3ex}
	u_1(0,t)=r_1u_2(0,t), \quad	u_2(1,t)=r_2u_1(1,t).
\end{equation}
To fulfill the assumptions of Theorem \ref{periodic}, we have  to show that, for appropriately chosen~$\alpha$,  $r_1$, and $r_2$, the system (\ref{eq:1ex})--(\ref{eq:3ex}) is exponentially stable in $L^2$ and that the  dissipativity conditions (\ref{G_i}) are satisfied.
Indeed, the corresponding eigenvalue problem  
reads
\begin{equation}
\label{exampleevp}
\begin{array}{l}
	v_1'-\mu v_1=0, \quad	v_2'+(\mu-\alpha) v_2 + v_1 = 0,\\
	v_1(0)=r_1v_2(0), \quad	v_2(1)= r_2v_1(1),
\end{array}
\end{equation}
where $\mu$ is a spectral parameter. 
By direct calculations, we can see that the system (\ref{exampleevp}) is equivalent to
\begin{eqnarray*}
	& &v_1(x)=ce^{\mu x},\quad v_2(x)=\frac{c}{\alpha-2\mu}\left(e^{\mu x} +\left[(\alpha-2\mu)r_1^{-1}-1\right]e^{(\alpha-\mu)x}\right), \\
	& & r_2(\alpha-2\mu)-1=\left[(\alpha-2\mu)r_1^{-1}-1\right]	e^{\alpha-2\mu},
\end{eqnarray*}
where $c=v_1(0)$ is a nonzero complex constant.
Setting $\alpha-2\mu=\xi+i\eta$ with $\xi \in \mathbb R$ and (without loss of generality) $\eta>0$ and looking at the real and the imaginary parts of the last equality, we get
\begin{equation}\label{eta}
\eta=\sqrt{\frac{(\xi- r_1)^2e^{2\xi}-(\xi r_2-1)^2r_1^2}{r_1^2r_2^2-e^{2\xi}}}
\end{equation}
and
\begin{equation}
\label{chareq}
\sin \eta=\frac{(1-r_1r_2)\eta r_1}{{e^\xi}\left[(\xi -r_1)^2+\eta^2\right]}.
\end{equation}
It is easy to see that,  besides of the zero solution, the equation (\ref{chareq}) has  a number of
solutions $\xi_j$ lying on a bounded  interval, say $|\xi_j|\le\bar{\xi}$ for some $\bar\xi>0$. This follows from the formulas (\ref{eta}), (\ref{chareq}) and from the fact that 
$$
\frac{(\xi- r_1)^2e^{2\xi}-(\xi r_2-1)^2r_1^2}{r_1^2r_2^2-e^{2\xi}}\to-\infty\ \mbox{ as }\ \xi\pm\infty.
$$
Hence,  
the  eigenvalues of (\ref{exampleevp}) are of the form
\begin{equation}\label{spectrum}
\mu_j^\pm(\alpha)=\frac{1}{2}\left(\alpha-\xi_j\pm i\sqrt{\frac{(\xi_j- r_1)^2e^{2\xi_j}-(\xi_j r_2-1)^2r_1^2}{r_1^2r_2^2-e^{2\xi_j}}}\right).
\end{equation} 
If $\alpha$ is chosen to satisfy the inequality $\alpha<-\bar \xi$, then the real parts of all eigenvalues
are negative.  

We now use the spectral mapping theorem (see \cite[Theorem 2.1]{Lichtner})  
providing a spectral criterion of  the  exponential stability of    $C_0$-semigroups in $L^2$. Combined with \cite{CoronBastin}, 
this criterion states that the 
spectrum of (\ref{exampleevp})  determines the spectrum of the semigroup generated by the original problem (\ref{eq:1ex})--(\ref{eq:3ex}). 
As the problem (\ref{eq:1ex})--(\ref{eq:3ex}) is autonomous, from  Theorem~\ref{evol} we conclude that the operator 
$\mathcal A$ associated with (\ref{eq:1ex})--(\ref{eq:3ex}) (and defined by (\ref{A})--(\ref{D(A)})
below) generates a $C_0$-semigroup  on $L^2$. 
This means, in particular, that the operator $\mathcal A$ is  closed.
It is known from \cite[Theorem 6.29, p.~187]{kato} that the spectrum of the operator $\mathcal A$ consists of  finitely or
countably many eigenvalues, each of finite multiplicity. Furthermore,  it has no finite
limit points. This implies that the whole spectrum of $\mathcal A$ is given by the formula (\ref{spectrum}). Due to the
spectral mapping theorem, the problem (\ref{eq:1ex})--(\ref{eq:3ex}) is exponentialy  stable in  $L^2$
for all $\alpha<-\bar \xi$.

It remains to check the condition (\ref{G_i}). 
We suppose, additionally to  $\alpha<-\bar \xi$, that the boundary coefficients $r_1$ and $r_2$ fulfill the conditions
$|r_1|<1$ and $|r_2|e^{-\alpha}<1$.
Note first that in the autonomous setting   the condition  (\ref{G_i}) is the same for $i=0,1,2$. 
It, therefore, suffices  to check this condition  for $i=0$. In this case, (\ref{G_i}) follows from the inequalities, which are true
for any $\psi\in BC(\mathbb R;\mathbb R^n)$ with $\|\psi\|_{BC(\mathbb R;\mathbb R^n)}=1$ and
 for all $t\in\mathbb R$, namely
$$
\left|[G_0\psi]_1(t)\right|\le |r_1|<1,\quad\left|[G_0\psi]_2(t)\right|\le |r_2|e^{-\alpha}<1.
$$

This example also shows that, if  the absolute values of the  coefficients of the hyperbolic system are large, then the dissipativity conditions are ensured if  the boundary coefficients $r_1$ and $r_2$ are small enough.

\!{\it Example 2.}	We consider a time-periodic problem for the following nonautonomous decoupled  $2\times 2$-hyperbolic system:
	\begin{eqnarray}
\partial_t u_1 + \left(1-\frac{1}{4}\sin{t}\right)\partial_x u_1 + 3 u_1 = 0, \quad  \partial_t u_2 - (3+x)\partial_x u_2  = 0 \label{ex1}
	\end{eqnarray}
endowed with the boundary conditions (\ref{eq:3ex}). 
The aim of this example is, first, to check  the main solvability conditions of this problem
 in the nonautonomous setting and, second, to show  that at least one of the boundary coefficients $r_1$ and $r_2$ can
be  large. To this end, we use the Lyapunov approach.
Let
	$V(t) = \int_0^1 \left(u_1^2(x,t) + u_2^2(x,t)\right) d x.$ The derivative of $V(t)$
	along the trajectories of the system (\ref{ex1}), (\ref{eq:3ex}) is computed directly as follows:
	\begin{eqnarray*}
		\frac{dV(t)}{dt} &=& 
		 2\int_0^1 \left[u_1\left(\left(\frac{1}{4}\sin{t}-1\right)\partial_x u_1- 3u_1\right) 
		+ (3+x)u_2\partial_x u_2\right] d x \\
		& = &\left[\left(\frac{1}{4}\sin{t}-1\right) + 4 r_2^2\right]u_1^2(1,t) + 
		\left[-3 +r_1 ^2\left(1-\frac{1}{4}\sin{t}\right) \right] u_2^2(0,t)\\ && - \int_0^1\left(6 u_1^2 + u_2^2\right) d x.
	\end{eqnarray*}
	Putting $r_1 = \sqrt{12/5}>1$ and $r_2 = \sqrt{3}/4<1$, we have
	$$\left[\left(\frac{1}{4}\sin{t}-1\right)+ 4 r_2^2\right] u_1^2(1,t) + 
	\left[ -3 +r_1^2 \left(1-\frac{1}{4}\sin{t}\right)\right]u_2^2(0,t) \le 0.$$ 
	It follows that
	$\frac{dV(t)}{dt} \le  -V(t)$ and, consequently,
	$$\|u_1(\cdot,t)\|_{L^2(0,1)}^2 + \|u_2(\cdot,t)\|_{L^2(0,1)}^2 \le 
	e^{-(t-s)}\left(\|u_1(\cdot,s)\|_{L^2(0,1)}^2 + \|u_2(\cdot,s)\|_{L^2(0,1)}^2\right).$$
	 for $t\ge s.$ In other words, the system (\ref{ex1}), (\ref{eq:3ex}) is exponentially stable in $L^2$.

The  dissipativity conditions  (\ref{G_i}) are fulfilled, which can again be seen by the direct computation. It suffices to note that, for all $\psi\in BC(\mathbb R;\mathbb R^n)$ with $\|\psi\|_{BC(\mathbb R;\mathbb R^n)}=1$, for all $t\in\mathbb R$ and  for $i = 0,1,2$,  we have the following inequalities:
$$
\left|[G_i\psi]_1(t)\right|\le \sqrt{\frac{12}{5}}\exp{\left(-\frac{12}{5} + i\frac{4}{9}\right)} <1\quad\mbox{and}\quad
\left|[G_i\psi]_2(t)\right|\le  \frac{\sqrt{3}}{4}.
$$

\subsubsection{How to check the exponential stability  conditions: overview of known results}
  Examples $1$ and $2$  show that the dissipativity conditions (\ref{G_i}) are relatively easy to check
by direct computation.
For the exponential stability assumption, (sufficient) conditions
are desirable.  In this subsection we  survey some known results providing  a number  of such conditions. Note that those
results concern mostly the autonomous setting.

As we have already seen in Example 1, a useful  
tool is provided by the  spectral mapping theorem giving powerful {\it spectral criteria}  for deciding whether the system is exponentially stable. 
The applicability of the spectral criterion to our problem  follows from \cite{CoronBastin,Lichtner,Neves,Renardy}.
Another, {\it resolvent criterion} for the uniform exponential stability 
of    $C_0$-semigroups on Hilbert spaces
is given by 
 \cite[Theorem 1.11, p. 302]{Engel}. This theorem states that a $C_0$-semigroup $T(\cdot)$ on a Hilbert space $H$
 is uniformly exponentially stable if and only if there exists a constant
 $M >0$ such that
 $$
 \|(A-\mu I)^{-1}\|_{\mathcal{L}(H)}< M\ \mbox{ for all } \  \mu\in\mathbb C \  \mbox{ with } \   \Re\mu>0,
 $$
 where $A$ is a generator of $T(\cdot)$.
 
A group of sufficient conditions for the exponential stability 
  stems from  the  
{\it Lyapunov approach}
(see, e.g. \cite{BCbook,BCN,CoronBastin,DGL,DGL10,HSh1,HSh}). Being
natural  from the physical point of view, this approach provides   physically motivated conditions  on the data.
In 
\cite{BCbook} the authors deal with the $H^2$-exponential stability and provide two sufficient conditions, the so-called ``internal condition'' and  ``boundary condition''. The former is
based on the energy-like Lyapunov  function and the latter is formulated in terms of the boundary coefficients. 
Most results on the Lyapunov approach known in the literature deal with autonomous problems, but, as we showed in Example 2,  this approach is also applicable to the nonautonomous setting.

Finally,  we mention the so-called {\it backstepping approach}. It provides a suitable
transformation of the original system (by shifting all eigenvalues
accordingly) to the so-called ``target system'', which is more convenient to investigate the exponential stability, see e.g.~\cite{KrSm}.

\section{Proof of Theorem \ref{evol}}
\renewcommand{\theequation}{{\thesection}.\arabic{equation}}
\setcounter{equation}{0}

The proof will be established through a sequence of Lemmas \ref{km}--\ref{C1} presented below. Importantly, these lemmas extend beyond the homogeneous case to  the nonhomogeneous scenario,  which will be used later.

\begin{defn}\rm
	A continuous function $u:\Pi_s\to \mathbb R^n$ is called a  {\it piecewise continuously differentiable  solution}
	to the problem 
	(\ref{eq:1}), (\ref{eq:2}), (\ref{eq:in}) if it is  continuously differentiable  in $\Pi_s$ almost everywhere (excepting at most  countable number of
	characteristic curves of  (\ref{eq:1}), denoted by $J_s$) and satisfies the problem  
	(\ref{eq:1}), (\ref{eq:2}), (\ref{eq:in}) everywhere in $\Pi_s\setminus J_s$.
\end{defn}
Note that the derivatives of a piecewise continuously differentiable  solution
restricted to a compact subset of $\Pi_s$
have not more than finite number of discontinuities of first order  on~$J_s$.

\begin{lemma}\label{km}\cite{ijdsde}
	Let    $a_j, b_{jk}, f_j\in C_x^1(\Pi)$ for all $j,k\le n$.
	 Assume that the condition (\ref{eq:h1})  is fulfilled and   the initial function $\varphi$ belongs to $C^1([0,1];\mathbb R^n)$ and
 fulfills the zero order compatibility conditions
	(\ref{zero}).
	Then the problem (\ref{eq:1}), (\ref{eq:2}), (\ref{eq:in}) has a unique piecewise continuously differentiable  solution in $\Pi_s$.
	\end{lemma} 
 
 \begin{lemma}\label{H1} \cite[Lemma 4.2]{KRT} 
	Let $a_{j}$, $b_{jk},$ $f_j\in BC^1(\Pi)$.
	Assume that the condition (\ref{eq:h1})  is fulfilled and  the initial function  $\varphi$  belongs to $C^1([0,1];\mathbb R^n)$ and
	  fulfills the zero order compatibility conditions
	(\ref{zero}). Then 
	there exist constants $K_1$ and $\nu_1$ such that
	the  piecewise continuously differentiable
	solution  $u$ to the problem  (\ref{eq:1}), (\ref{eq:2}), (\ref{eq:in})
	fulfills the  estimate
\begin{equation}\label{eq:apr4}
	\begin{array}{ll}
		\|u(\cdot,t)\|_{H^1\left((0,1);\mathbb R^n\right)}+
		\|\partial_tu(\cdot,t)\|_{L^2\left((0,1);\mathbb R^n\right)}\\ [2mm]
		\qquad \le  K_1 e^{\nu_1(t-s)}\left(\|\varphi\|_{H^1((0,1);\mathbb R^n)}+
		\|f\|_{BC^1\left(\mathbb R, L^2\left((0,1);\mathbb R^n\right)\right)}
		\right)
	\end{array}
\end{equation}
	for all  $t\ge s$.
\end{lemma}

 \begin{lemma}\label{C1}
 		Let $a_{j}$, $b_{jk},$ $f_j\in BC^1(\Pi)$ for all $j,k\le n$.
 Assume that the condition (\ref{eq:h1} )  is fulfilled and the function  $\varphi$ belongs to $ H^1\left((0,1);\mathbb R^n\right)$ and fulfills the zero order compatibility conditions (\ref{zero}).
 	Then the 
 	$L^2$-generalized solution $u$ to the problem  (\ref{eq:1}), (\ref{eq:2}), (\ref{eq:in})
 	belongs to  $C([s,t]; H^1\left((0,1);\mathbb R^n\right))$ and to $C^1([s,t]; L^2\left((0,1);\mathbb R^n\right))$ for any $t\ge s$.
\end{lemma}

\begin{proof}
	Define $\widetilde\varphi(x) = \varphi(0) + x(\varphi(1) - \varphi(0))$ for $x \in [0,1].$
	Note that $\varphi - \widetilde\varphi \in H^1_0((0,1);\mathbb R^n)$, see also \cite[p. 259]{Ev}. Then
	there exists a sequence
	$\varphi^l_0\!\in\! C_0^\infty([0,1];\mathbb R^n)$ approaching $\varphi - \widetilde\varphi$ in $H^1((0,1);\mathbb R^n)$.
Hence, the sequence $\varphi^l = \varphi^l_0 + \widetilde\varphi$  approaches $\varphi$ in $H^1((0,1);\mathbb R^n)$.
	
	By  Lemma \ref{H1}, for given $l\ge 1$,
	the piecewise continuously differentiable
	solution $u^l$ to the problem  (\ref{eq:1}), (\ref{eq:2}), (\ref{eq:in}) with $\varphi^l$
	in place of $\varphi$ satisfies the estimate  (\ref{eq:apr4}) with $u=u^l$ and $\varphi=\varphi^l$.
	This entails the convergence
	$$\max\limits_{s\le\theta\le t}\|u^m(\cdot,\theta)-u^l(\cdot,\theta)\|_{H^1\left((0,1);\mathbb R^n\right)}+
	\max\limits_{s\le\theta\le t}\|\partial_tu^m(\cdot,\theta)-\partial_tu^l(\cdot,\theta)\|_{L^2\left((0,1);\mathbb R^n\right)}\to 0 $$
	as $m,l\to\infty$ for each $t>s$. 
	Consequently, the sequence $u^l$ converges in the space
	$C([s,t];H^1\left((0,1);\mathbb R^n\right))\cap C^1([s,t];L^2\left((0,1);\mathbb R^n\right))$ to a function
	$$u\in C([s,t];H^1\left((0,1);\mathbb R^n\right))\cap 
	C^1([s,t];L^2\left((0,1);\mathbb R^n\right)),$$ as desired.
\end{proof}

{\it Proof of Theorem \ref{evol}}:
Now we switch to the homogeneous problem (\ref{eq:1u}), (\ref{eq:2}), (\ref{eq:in}). 
By Lemma \ref{C1}, 	the 
$L^2$-generalized solution  $u$ to   (\ref{eq:1u}), (\ref{eq:2}), (\ref{eq:in}) belongs to 
$C([s,t],H^1\left((0,1);\mathbb R^n\right))$ and to $ C^1([s,t],L^2\left((0,1);\mathbb R^n\right))$.
Using the definition of $U(t,s)$ (see Theorem \ref{evol0}), this implies the validity of (\ref{cap}). In particular, the mapping 
 $t\in \mathbb R \mapsto U(t,s)\varphi\in L^2\left((0,1);\mathbb R^n\right)$ is continuously differentiable  for all $t\ge s$. 
Furthermore, the continuous differentiability extends to the mapping $s\in \mathbb R \mapsto U(t,s)\varphi\in L^2\left((0,1);\mathbb R^n\right)$  for all $t\ge s$.
This follows from a standard argument based on the evolution property of $U(t,s)$ (the property $(a)$  in Theorem~\ref{evol0}).

\section{Proof of Theorem \ref{lin-smooth}}
\subsection{Abstract problem setting in $L^2$}\label{abstract}
\renewcommand{\theequation}{{\thesection}.\arabic{equation}}
\setcounter{equation}{0}

We now  reformulate the problem  (\ref{eq:1}), (\ref{eq:2}),  (\ref{eq:in}) as an abstract evolution equation in $L^2\left((0,1);\mathbb R^n\right)$. To this end, we introduce a one-parameter family of 
operators $\mathcal A(t)$ (where $t\in \mathbb R$ serves as the parameter) mapping  
$L^2\left((0,1);\mathbb R^n\right)$ to itself, by
\begin{equation}\label{A}
[\mathcal A(t)u](x)=\left(-a(x,t)\frac{\partial}{\partial x} - b(x,t)\right)u,
\end{equation}
with the domain 
\begin{equation}\label{D(A)}
\displaystyle
 D(\mathcal A(t))=\left\{v\in H^1((0,1);\mathbb R^n)\,:\,v_j(x_j)=
R_jv,\, j \le n\right\}\subset L^2((0,1);\mathbb R^n).
\end{equation}
Note that $D(\mathcal A(t))=D(\mathcal A)$ is independent of $t$.

Writing  $u(t)$ and  $f(t)$, we will mean  bounded and continuous maps from $\mathbb R$ to $L^2((0,1);\mathbb R^n)$
defined by $[u(t)](x)=u(x,t)$ and $[f(t)](x)=f(x,t)$, respectively.
In this notation, the  abstract formulation of the problem (\ref{eq:1}), (\ref{eq:2}), (\ref{eq:in}) takes the form
\begin{eqnarray}
&&\frac{d}{dt}u=\mathcal A(t)u +f(t), \label{unperturb}\\\
&& u(s)=\varphi\in L^2((0,1);\mathbb R^n).\label{unperturbb}
 \end{eqnarray}

\begin{defn} \rm
{\bf 1.}	 
A function $u\in C\left([s,\infty),L^2((0,1);\mathbb R^n)\right)$ is called a {\it classical solution to
		the abstract Cauchy problem (\ref{unperturb})--(\ref{unperturbb})} if
	$u$ is continuously differentiable in $L^2((0,1);\mathbb R^n)$ for $t>s$, $u(t)\in D(\mathcal A)$ for $t>s$, and
	(\ref{unperturb})--(\ref{unperturbb})  is satisfied in $L^2((0,1);\mathbb R^n)$.\\
	{\bf 2.}
A function $u\in BC\left(\mathbb R,L^2((0,1);\mathbb R^n)\right)$ is called a {\it bounded classical solution to
	the abstract equation (\ref{unperturb})} if
$u$ is continuously differentiable in $t$ with values in $L^2((0,1);\mathbb R^n)$,  $u(t)\in D(\mathcal A)$ for all $t\in\mathbb R$,  and
(\ref{unperturb})  is satisfied in $L^2((0,1);\mathbb R^n)$.	
\end{defn}

We will use the equivalence between the original and the abstract problem settings, as provided by the following theorem.
\begin{thm}\cite[Theorem 4.1]{KRT}\label{distr-clas} 
Suppose that $a_j,b_{jk},f_j\in BC^1(\Pi)$ and the condition (\ref{eq:h1})  is fulfilled.
If  $\varphi$ belongs to $ D(\mathcal A)$ and $u(x,t)$ is the 
$L^2$-generalized solution to the problem (\ref{eq:1}), (\ref{eq:2}), (\ref{eq:in}), then 
  the function
$u(t)$ such that $[u(t)](x):=u(x,t)$,
is a classical solution  to the abstract Cauchy problem (\ref{unperturb})--(\ref{unperturbb}).
Vice versa, if $u(t)$ is a classical solution  to the abstract Cauchy problem (\ref{unperturb})--(\ref{unperturbb}),
then $u(x,t):=[u(t)](x)$ is the $L^2$-generalized solution to the problem
(\ref{eq:1}), (\ref{eq:2}), (\ref{eq:in}).
\end{thm}

\begin{thm} \label{u_*}
	Assume that $a_j, b_{jk}, f_j \in BC^1(\Pi)$ for all $j,k \le n$ and
the evolution operator $U(t,s)$ satisfies the exponential estimate (\ref{est1}).
Then the equation (\ref{unperturb}) has a unique bounded  classical solution 
\begin{eqnarray} \label{boundsol}
v^*(t) = \int_{-\infty}^t U(t,s)f(s) d s.
\end{eqnarray}
Moreover, the following a priori estimate is fulfilled:
\begin{eqnarray} \label{boundsol1}
\|v^*\|_{BC\left(\mathbb R;L^2((0,1);\mathbb R^n)\right)} \le \frac{M}{\alpha}\|f\|_{BC\left(\mathbb R;L^2((0,1);\mathbb R^n)\right)},
\end{eqnarray}
 where $M$ and $\alpha$ are  as in (\ref{est1}).
\end{thm}
\begin{proof}
	As the boundedness estimate (\ref{boundsol1}) immediately follows from the  representation formula (\ref{boundsol}) and 
	the exponential stability estimate (\ref{est1}),
	our objective is  simplified to proving (\ref{boundsol}).
	We divide the proof of (\ref{boundsol}) into two claims.
	
	\begin{claim}
		 There exists  a (sufficiently large)  $\lambda>0$  such that, for any $t\in\mathbb R$,
	the operator ${\mathcal A} (t)+ \lambda I$	is invertible.
		The inverse, denoted as $({\mathcal A} (t)+ \lambda I)^{-1}$, maps  $L^2((0,1);\mathbb R^n)$ to $H^1((0,1);\mathbb R^n)$ and is both continuous and uniformly bounded in $t\in \mathbb R$. Furthermore, 
	  \begin{equation}\label{M0}
	 \left\| ({\mathcal A} (t)+ \lambda I)^{-1}g\right\|_{L^2((0,1);\mathbb R^n)}\le 
	 \frac{M_0}{\sqrt{\lambda}}\|g\|_{L^2((0,1);\mathbb R^n)}
	\end{equation}
	  and
	  \begin{equation}\label{M1}
	\left\| ({\mathcal A} (t)+ \lambda I)^{-1}g\right\|_{H^1((0,1);\mathbb R^n)}\le 
	M_0\sqrt{\lambda}\|g\|_{L^2((0,1);\mathbb R^n)}
	\end{equation} 
	   for all $g\in L^2((0,1);\mathbb R^n)$ and some $M_0>0$ independent of $\lambda$, $t$,  and $g$.
		\end{claim}
\begin{subproof}
Set $u(x,t) = (u^1(x,t), u^2(x,t)),$ where
$$u^1(x,t) = (u_1(x,t),\dots,u_m(x,t)), \quad u^2(x,t) = (u_{m+1}(x,t),\dots,u_n(x,t)).$$
Moreover, set
$$
R_{11}\!=\!(r_{jk})_{j,k=1}^m, \ R_{12}\!=\!(r_{jk})_{j=1,k=m+1}^{m,n}, R_{21}\!=\!(r_{jk})_{j=m+1,k=1}^{n,m},\ R_{22}\!=\!(r_{jk})_{j,k=m+1}^n,
$$
where $r_{jk}$ are the reflection boundary coefficients introduced in (\ref{eq:R}).

The claim will be proved if we show  that there exists  $\lambda>0$ such that, for each $t\in\mathbb R$ and $g\in L^2((0,1);\mathbb R^n)$, the linear ODE system 
\begin{eqnarray} \label{linop1}
\frac{d u}{d x} +a^{-1}(x,t) (b(x,t) + \lambda I)u = g(x)
\end{eqnarray}
with the boundary conditions 
\begin{eqnarray} \label{linop2}
u^1(0,t) = R_{11}u^1(1,t) + R_{12}u^2(0,t), \quad u^2(1,t) = R_{21}u^1(1,t) + R_{22}u^2(0,t)
\end{eqnarray}
 has a unique  solution $u(\cdot,t)\in H^1((0,1);\mathbb R^n)$, which is continuous and uniformly bounded in 
 the parameter $t$.
Rewrite  (\ref{linop1}) in the form of an operator equation in  $L^2((0,1);\mathbb R^n)$, specifically as
\begin{eqnarray*} 
{\mathcal A}_1(t)u +  {\mathcal A}_2(t)u  = g, 
\end{eqnarray*}
where 
$$
\begin{array} {rcl}
[\mathcal A_1(t)u](x)  &=&\displaystyle  \frac{d u}{d x}  +\lambda a^{-1}(x,t) u, \\ [3mm]
D(\mathcal A_1)&=&\Bigl\{v\in H^1((0,1);\mathbb R^n)\,:\,v^1(0) = R_{11}v^1(1) + R_{12}v^2(0), \\ [3mm]
&&\ \  v^2(1) = R_{21}v^1(1) + R_{22}v^2(0) \Bigr\}
\end{array}
$$
and 
$$
\begin{array} {rcl}
[\mathcal A_2(t)u](x)  &=& a^{-1}(x,t) b(x,t)u,\\ [3mm]
D(\mathcal A_2)&=&L^2((0,1);\mathbb R^n).
\end{array}
$$

To establish the desired solvability of  (\ref{linop1})--(\ref{linop2}), we begin by proving the invertibility of ${\mathcal A}_1$. For this purpose, let  $g=(g^1,g^2)$ be an arbitrary vector with  $g^1\in L^2((0,1);\mathbb R^m)$ and
$g^2\in L^2((0,1);\mathbb R^{n-m})$.  We then consider the two decoupled systems:
\begin{eqnarray*}
\frac{du^1}{dx} =\lambda A_1(x,t) u^1 + g^1(x)
\end{eqnarray*}
and
\begin{eqnarray*}
 \frac{du^2}{dx} =\lambda A_2(x,t) u^2 + g^2(x),	
\end{eqnarray*}
where $A_1 = -diag(a_1^{-1},\dots,a_m^{-1})$ and $A_2 = -diag(a_{m+1}^{-1},\dots,a_n^{-1}).$	
Let 
$$V_1(x,\xi;t,\lambda) = \exp \int_\xi^x \lambda A_1(\xi,t) d\xi\quad \mbox{ and }\quad
V_2(x,\xi;t,\lambda) = \exp \int_\xi^x \lambda A_2(\xi,t) d\xi$$
 be fundamental matrices of the corresponding homogeneous systems, respectively.  Notice that the entries of the matrices $V_j$
 are continuous in $t$. This yields the solution formulas
\begin{eqnarray}  \label{linop6}
\begin{array}{ll}
\displaystyle u^1(x;t,\lambda) = V_1(x,0;t,\lambda)u^1(0;t,\lambda) + 
 \int_{0}^x V_1(x,\xi;t,\lambda) g^1(\xi) d\xi, \\[3mm]
\displaystyle u^2(x;t,\lambda) = V_2(x,1;t,\lambda)u^2(1;t,\lambda) +  \int_{1}^x V_2(x,\xi;t,\lambda) g^2(\xi) d\xi.
\end{array}
\end{eqnarray}
 By our assumption, $a_j > 0$ for $j\le m$ and $a_j < 0$ for $j >m.$ Then we easily derive the bounds
\begin{eqnarray}  \label{linop7}
\|V_1(x,\xi;t,\lambda)\| \le  e^{-a\delta\lambda(x - \xi)} \quad \mbox{and} \quad
 \|V_2(x,\xi;t,\lambda)\| \le  e^{\delta \lambda(x - \xi)} \quad  \mbox{for} \  x \ge \xi,
 \end{eqnarray}
where the constant $\delta$ is positive and is  independent of $\lambda$, $x$, $\xi$, and $t$.
Accordingly to (\ref{linop6}), 
\begin{eqnarray} \label{linop8}
 u^1(1) = V_1 u^1(0) +  J_1, \quad
 u^2(0) = V_{2} u^2(1) + J_2,
\end{eqnarray}
where, for brevity, we omitted the explicit dependence on $t$ and $\lambda$ in the  following notation: $u^1(x)=u^1(x;t,\lambda)$, $u^2(x)=u^2(x;t,\lambda)$, 
$V_1 = V_1(1,0;t,\lambda),$ $V_2 = V_2(0,1;t,\lambda)$, and
$$
 J_1 = \int_0^1  V_{1}(1,\xi;t,\lambda)g^1(\xi) d\xi, \quad
 J_2 = -\int_0^1  V_{2}(0,\xi;t,\lambda)g^2(\xi) d\xi.$$
Substituting  (\ref{linop8}) into  (\ref{linop2}) implies the equalities
\begin{equation}\label{101}
\begin{array}{rcl}
(I -R_{11} V_1) u^1(0) - R_{12} V_2 u^2(1)& =& R_{11} J_1 + R_{12} J_2, \\[2mm]
 -R_{21} V_1 u^1(0) + (I - R_{22}V_{2}) u^2(1) &= &R_{21} J_1 + R_{22}J_2.
\end{array}
\end{equation}
Note that, due to (\ref{linop7}),  the following uniform in $t$ estimates are true:
\begin{equation}\label{100}
\|V_j\| \le e^{-\delta\lambda}\quad\mbox{and}\quad\|J_j\| \le  \frac{1}{ \sqrt{2\delta\lambda}}\|g\|_{L^2((0,1);\mathbb R^n)}\quad\mbox{for } j=1,2.
\end{equation}
 Consequently, we can choose  $\lambda>0$ such that the operator 
$I -R_{11} V_1$ is invertible, and the following estimate is fulfilled:
\begin{eqnarray*} 
	\| (I - R_{11} V_1)^{-1}\| \le (1 - \|R_{11}\| e^{-\delta\lambda})^{-1}.
\end{eqnarray*}
Hence,
$$ 
  u^1(0) = (I - R_{11} V_1)^{-1}R_{12} V_2 u^2(1) + (I - R_{11} V_1)^{-1} (R_{11} J_1 + R_{12} J_2)
  $$
  and, therefore,
  $$
  \begin{array}{ll}
  \left[(I - R_{22} V_2 - R_{21}V_1  (I - R_{11} V_1)^{-1}R_{12} V_2 \right] u^2(1)  \\[1mm]
  \hspace{10mm} = R_{21} J_1 + R_{22}J_2 + R_{21}V_1  (I - R_{11} V_1)^{-1}(R_{11} J_1 + R_{12} J_2).
\end{array}
$$
On account of (\ref{100}), we can choose a (potentially new) sufficiently large $\lambda$ such that the matrix
$(I - R_{22} V_2 - R_{21}V_1  (I - R_{11} V_1)^{-1}R_{12} V_2$ is invertible and, moreover, there exists 
a constant $M_1 > 0$  fulfilling the estimate
 $$\|\left[(I - R_{22} V_2 - R_{21}V_1  (I - R_{11} V_1)^{-1}R_{12} V_2 \right]^{-1}\| \le M_1$$
uniformly in $t\in\mathbb R$. 
Therefore, for every function $g$, the system (\ref{101}) has  a unique  solution $(u^1(0;t,\lambda),u^2(1;t,\lambda))$, which is continuous in $t$. Furthermore, 
$$\|u_1(0)\| + \|u_2(1)\| \le \frac{M_2}{\sqrt{\lambda}}\|g\|_{L^2((0,1);\mathbb R^n)},
	$$
where the constant $M_2$ is independent of $t$ and $\lambda$.

We have, therefore, proved that the operator ${\mathcal A}_1(t)$ is invertible. As it follows from (\ref{linop6}), $(\mathcal A_1(t))^{-1} g$ belongs to 
$H^1((0,1);\mathbb R^n)$ and is a continuous function in the parameter $t$. Moreover,   there exists $M_3>0$ such that $\|(\mathcal A_1(t))^{-1} g\|_{L^2((0,1);\mathbb R^n)} \le  \frac{M_3}{\sqrt{\lambda}}\|g\|_{L^2((0,1);\mathbb R^n)}$. 
Furthermore, since the operator $\mathcal A_2(t)$ is bounded and does not depend on $\lambda$, we have  $\| (\mathcal A_1(t))^{-1}{\mathcal A}_2(t)\|_{\mathcal L(L^2((0,1);\mathbb R^n))}\le \frac{M_4}{\sqrt{\lambda}}< 1$ for some $M_4>0$ and
for sufficiently large $\lambda.$ Summarizing, there exists  a constant $M_0>0$ independent of $\lambda$ and $t$ such that
\begin{align*}
&\left\|[{\mathcal A}_1(t) + {\mathcal A}_2(t)]^{-1}\right\|_{\mathcal L(L^2((0,1);\mathbb R^n))} \\ 
&\qquad\le \left\|[I + (\mathcal A_1(t))^{-1}{\mathcal A}_2(t)]^{-1}(\mathcal A_1(t))^{-1}\right\|_{\mathcal L(L^2((0,1);\mathbb R^n))} \\ 
& \qquad\le 
\left\|(\mathcal A_1(t))^{-1}\right\|_{\mathcal L(L^2((0,1);\mathbb R^n))} \sum_{j=0}^\infty \left\|(\mathcal A_1(t))^{-1}{\mathcal A}_2(t)\right\|_{\mathcal L(L^2((0,1);\mathbb R^n))}^j\le
\frac{M_0}{\sqrt{\lambda}}.
\end{align*}
To finish the proof of this claim, it remains to note that  the estimate (\ref{M1}) now easily follows from (\ref{linop1}).
\end{subproof}
\begin{claim}
    The function $v^*$ given by (\ref{boundsol}) is a bounded classical solution to the abstract 
		equation~(\ref{unperturb}).
	\end{claim}	
\begin{subproof}	
		Note that the integral in (\ref{boundsol}) exists due to the continuity of $U(t,s)f(s)$ in $s$
		and the exponential bound (\ref{est1}).
		
On the account of  Claim 1, we can
 rewrite (\ref{boundsol}) in the form
\begin{equation}\label{boundsol2}
	\begin{split}
 v^*(t) &= \int_{-\infty}^t U(t,s)({\mathcal{A}}(s) + \lambda I)({\mathcal{A}}(s) + \lambda I)^{-1}f(s)ds 
 \\
&  = \int_{-\infty}^t U(t,s){\mathcal{A}}(s)({\mathcal{A}}(s) +\lambda I)^{-1}f(s) ds +
\lambda \int_{-\infty}^t U(t,s)({\mathcal{A}}(s) +\lambda I)^{-1}f(s)ds \\
&  = - \int_{-\infty}^t \frac{\partial U(t,s)}{\partial s}({\mathcal{A}}(s) + \lambda I)^{-1}f(s) ds +
\lambda \int_{-\infty}^t U(t,s)({\mathcal{A}}(s) + \lambda I)^{-1}f(s) ds \\
& = -U(t,s)({\mathcal{A}}(s) + \lambda I)^{-1}f(s)\biggl|_{s = -\infty}^t 
+ \int_{-\infty}^t U(t,s)\frac{\partial }{\partial s}\left[({\mathcal{A}}(s) + \lambda I)^{-1}f(s)\right] ds   \\
& \quad + \lambda \int_{-\infty}^t U(t,s)({\mathcal{A}}(s) + \lambda I)^{-1}f(s) ds. 
\end{split}
\end{equation}
Here we used the property that,  if $u\in D(\mathcal A)$, then
$$
\frac{\partial}{\partial s}U(t,s)u = -U(t,s)\mathcal A(s)u.
$$
Since
\begin{equation}\label{ds}
\frac{\partial}{\partial s} \left[(\mathcal A(s) +\lambda I)^{-1}f(s)\right] =-
({\mathcal{A}}(s) +\lambda I)^{-1} \left(\frac{d{\mathcal{A}}(s)}{ds}({\mathcal{A}}(s) + \lambda I)^{-1}f(s) -
f^\prime(s)\right)
\end{equation}
and the function $\left(\frac{d\mathcal A(s)}{ds}({\mathcal{A}}(s) + \lambda I)^{-1}f(s) -
f^\prime(s)\right)$ is  continuous and bounded in~$s$,  the function $\frac{\partial}{\partial s} \left[(\mathcal A(s) +\lambda I)^{-1}f(s)\right]$
 is also continuous and bounded in $s$, with values in  $D(\mathcal A)$. 
 Further, taking into account the exponential stability estimate (\ref{est1}), 
 the function
$$ \int_{-\infty}^t U(t,s)\frac{\partial}{\partial s}\left[({\mathcal{A}}(s) + \lambda I)^{-1}f(s)\right] d s $$
is continuously differentiable in $t$, like in the proof of  \cite[Theorem 3.2, p. 197]{Krein}.
Moreover, the last integral in (\ref{boundsol2}) is
 continuously differentiable in~$t$, by the same argument.
 
  It remains to show that  $v^*$ satisfies (\ref{unperturb})
 in the classical sense. To this end,
 set
\begin{equation}\label{ftilde}
\widetilde f(s) =-
({\mathcal{A}}(s) +\lambda I)^{-1} \left(\frac{d{\mathcal{A}}(s)}{ds}({\mathcal{A}}(s) + \lambda I)^{-1}f(s) -
f^\prime(s)-\lambda f(s)\right).
\end{equation}
By (\ref{boundsol2}),  the function  $v^*(t)$ reads
\begin{equation}\label{u*}
v^*(t) =  -({\mathcal{A}}(t) + \lambda I)^{-1}f(t) +\int_{-\infty}^t U(t,s)\widetilde f(s) ds.
\end{equation} 
 As $\widetilde f(s) \in D(\mathcal A)$, then the right hand side of (\ref{boundsol}) is continuously differentiable 
 and, due to (\ref{ds}), it holds
 \begin{eqnarray*}
 & & 	\displaystyle\frac{d v^*(t)}{dt}=\displaystyle
 	({\mathcal{A}}(t) +\lambda I)^{-1} \left(\frac{d{\mathcal{A}}(t)}{dt}({\mathcal{A}}(t) + \lambda I)^{-1}f(t) -
 	f^\prime(t)\right)+\widetilde f(t) \\
& & \displaystyle +\int_{-\infty}^t \frac{\partial U(t,s)}{\partial t}\widetilde f(s)d s
 = \lambda ({\mathcal{A}}(s) +\lambda I)^{-1}f(t)+\int_{-\infty}^t \frac{\partial U(t,s)}{\partial t}\widetilde f(s) d s.
  \end{eqnarray*}
 On the other hand,
 $$
 \begin{array}{rcl}
 	 \mathcal A(t) v^*(t)+f(t)&=&\displaystyle-\mathcal A(t)({\mathcal{A}}(t) + \lambda I)^{-1}f(t) +\mathcal A(t)\int_{-\infty}^t U(t,s)\widetilde f(s) ds+f(t)\\ [3mm]
 	 &=&\displaystyle	\lambda ({\mathcal{A}}(s) +\lambda I)^{-1}f(t)+\int_{-\infty}^t \frac{\partial U(t,s)}{\partial t}\widetilde f(s) d s,
 \end{array} $$
 as desired. 
\end{subproof}
 The proof of  Theorem \ref{u_*} is complete. 
\end{proof}
 
 \subsection{Proof of Part ${\bf (i)}$: $BC$-solutions}\label{BC0}
 
 First show that  the solution $v^*$ given by (\ref{boundsol}) belongs to 
 $BC(\mathbb R,H^1((0,1);\mathbb R^n))$. Since all the coefficients 
of the initial system and their first order derivatives are globally bounded in~$\Pi$, then combining
(\ref{M0}) with (\ref{ftilde}) yields that
\begin{equation}\label{M5}
 \| \widetilde f \|_{BC(\mathbb R;H^1((0,1);\mathbb R^n))} \le 
 M_5 \|f\|_{BC^1\left(\mathbb R;H^1((0,1);\mathbb R^n)\right)} 
\end{equation}
for a constant $M_5\!>\!0$ independent of $f$.
Introduce an operator $\mathcal A_d: H^1((0,1);\mathbb R^n)\to L^2((0,1);\mathbb R^n)$ defined by
$$
\mathcal A_du=\left(\partial_xu_1,\partial_xu_2,\dots, \partial_x u_n\right).
$$
Applying this operator  to both sides of (\ref{u*}) gives the equality
\begin{equation}\label{h1}
\mathcal  A_dv^*(t) =  -\mathcal  A_d({\mathcal{A}}(t) + \lambda I)^{-1}f(t) +\int_{-\infty}^t \mathcal  A_dU(t,s)\widetilde f(s) d s.
\end{equation}
The first summand in the right-hand side can  easily be estimated  using (\ref{M1}), specifically
\begin{equation}\label{h2}
	\begin{split}
		&	\|\mathcal A_d({\mathcal{A}}(t)+ \lambda I)^{-1}f(t)\|_{L^2((0,1);\mathbb R^n)}  \\  
&\qquad	\le \|\mathcal A_d\|_{\mathcal L(H^1,L^2)} \left\| ({\mathcal A} (t)+ \lambda I)^{-1}\right\|_{\mathcal L(L^2,H^1)}\|f(t)\|_{L^2((0,1);\mathbb R^n)} \\  
&\qquad
\le
M_6\|f(t)\|_{L^2((0,1);\mathbb R^n)},
	\end{split}
\end{equation}
where $M_6>0$ is independent of $t$ and $f$.
For the second summand, we use both the estimate (\ref{M5}) and the  assumption about the uniform exponential stability of~$U$
in~$H^1$.  We, therefore,  derive  the bound
\begin{equation}\label{h3}
	\begin{split}
&\displaystyle\int_{-\infty}^t\left\| \mathcal A_dU(t,s)\widetilde f(s)\right\|_{L^2(\Pi;\mathbb R^n)} d s  \\ &
\qquad\le
\int_{-\infty}^t\left\| \mathcal A_d\right\|_{\mathcal L(H^1,L^2)}\bigl\| U(t,s)\widetilde f(s)\bigr\|_{H^1(\Pi;\mathbb R^n)} d s \\
&\qquad\displaystyle
\le M\int_{-\infty}^t e^{-\alpha(t-s)}\bigl\|\widetilde f(s)\bigr\|_{H^1(\Pi;\mathbb R^n)}d s\le
M_7\|f\|_{BC^1(\mathbb R;H^1((0,1);\mathbb R^n))}.
	\end{split}
\end{equation}
Combining the  estimates
(\ref{h1})--(\ref{h3}) results in the following bound:
\begin{eqnarray*} 
	\|v^*\|_{BC\left(\mathbb R;H^1((0,1);\mathbb R^n)\right)} \le 
	M_8\|f\|_{BC^1\left(\mathbb R;H^1([0,1];\mathbb R^n)\right)},
\end{eqnarray*}
being true for some $M_8>0$ independent of $f$. 
Due to the continuous embedding of $H^1(0,1)$ into $C([0,1])$,  we have $v^*\in BC\left(\mathbb R;C([0,1];\mathbb R^n)\right)$.
Moreover,
there exist   constants $M_9$ and $M_{10}$
such that
\begin{eqnarray*}
	\|v^*\|_{BC\left(\mathbb R;C([0,1];\mathbb R^n)\right)}\le M_9\|v^*\|_{BC\left(\mathbb R;H^1((0,1);\mathbb R^n)\right)} \le M_{10}\|f\|_{BC^1\left(\mathbb R;H^1([0,1];\mathbb R^n)\right)}.
\end{eqnarray*}
 Hence, the function $u^*(x,t)=[v^*(t)](x)$
belongs to $BC(\Pi; \mathbb R^n)$ and satisfies the a priori estimate
\begin{eqnarray} \label{h5}
	\|u^*\|_{BC(\Pi;\mathbb R^n)} \le  M_{10}\left(\| f\|_{BC^1(\Pi;\mathbb R^n)} + \| \partial^2_{xt}f\|_{BC(\Pi;\mathbb R^n)} \right).
\end{eqnarray}
Note that this estimate  is far from being optimal in the
$BC$-space, as it demands too high regularity from
the right-hand side $f$.  This
lack of optimality stems from our  method of obtaining  (\ref{h5}). Nevertheless,
  our final a priory estimate in the $BC^2$-space will show more optimal correspondence between the regularities 
  of the right hand sides and the solutions.

  Now, since $v^*$ is a bounded  classical solution to (\ref{unperturb}),
the function $u^*$ satisfies (\ref{eq:1}), (\ref{eq:2}) almost everywhere, where the partial derivatives in
(\ref{eq:1}) are understood
in the distributional sense and the boundary conditions (\ref{eq:2}) pointwise.
This means, equivalently, that 
$u^*$ satisfies the following system of integral equations pointwise:
\begin{align}\label{rep}
u^*_j(x,t)&= c_j(x_j,x,t)R_ju^*(\cdot,\omega_j(x_j))  \nonumber\\ 
&\quad - \displaystyle\int_{x_j}^x d_j(\xi,x,t)\biggl(
\sum_{k\not=j} b_{jk}(\xi,\omega_j(\xi))
u^*_k(\xi,\omega_j(\xi)) - f_j(\xi,\omega_j(\xi))\biggr)d\xi, \quad  j\le n,
\end{align}
(for details see, e.g. \cite[p. 4192--4193]{jee}). In this case, $u^*$ will be called a
{\it bounded continuous solution}
to (\ref{eq:1}), (\ref{eq:2}).

\subsection{Proof of Part ${\bf (i)}$: $BC^1$-solutions}\label{BC1}

The higher regularity of the $BC$-solution will be established using an 
argument analogous to that in \cite{jee}.

Introduce  operators
$C,D,F\in\mathcal L(BC(\Pi;\mathbb R^n))$ by
\begin{eqnarray*}
	[Cu]_j(x,t) & = & c_j(x_j,x,t)R_j u(\cdot,\omega_j(x_j,x,t)), \\[1mm]
	[Du]_j(x,t) & = &
	-\int_{x_j}^{x}d_j(\xi,x,t)\sum_{k\neq j} b_{jk}(\xi,\omega_j(\xi,x,t))u_k(\xi,\omega_j(\xi,x,t)) d\xi, \\[1mm]
	[Ff]_j(x,t) & = &  \int_{x_j}^{x} d_j(\xi,x,t) f_j(\xi,\omega_j(\xi,x,t)) d\xi
\end{eqnarray*}
In terms of those operators the system (\ref{rep})  reads as
\begin{eqnarray} \label{oper}
	u=Cu+Du+Ff.
\end{eqnarray}
After an iteration of this equation we  get
\begin{eqnarray} \label{oper2}
u=Cu+(DC + D^2)u+ (I + D)Ff.
\end{eqnarray}
 
 We now prove that
$I - C$ is a bijective operator from  $BC_t^1(\Pi; \mathbb R^n)$ to itself.
We are done if we  show that the system 
\begin{equation}\label{simpl}
u_j(x,t)=c_j(x_j,x,t)R_ju(\cdot,\omega_j(x_j,x,t))+h_j(x,t), \quad j\le n,
\end{equation}
is uniquely solvable in $BC^1_t (\Pi;\mathbb R^n)$ for any $h\in BC^1_t (\Pi;\mathbb R^n)$.
To achieve this, we consider the system obtained by evaluating (\ref{simpl})
at $x=0$ for $m<j\le n$ and at $x=1$ for $1\le j\le m$.
 This results in the following system of $n$ equations with respect to
$z(t) = (z_1(t),\dots,z_n(t))= (u_1(1,t),\dots, u_m(1,t), u_{m+1}(0,t),\dots,u_n(0,t))$:
\begin{equation}\label{simpl1}
\begin{array}{ll}
\displaystyle	z_j(t)= c_j(x_j,1 - x_j,t)\sum\limits_{k=1}^nr_{jk}z_k(\omega_j(x_j, 1 - x_j, t)),
	\quad  j\le n.
\end{array}
\end{equation}
On the account of (\ref{eq:G}), the system (\ref{simpl}) is uniquely solvable in $BC^1_t (\Pi;\mathbb R^n)$ if and only if
\begin{equation}\label{contr1}
I-G_0 \mbox{ is bijective from }   BC^1(\mathbb R;\mathbb R^n) \mbox{ to }   BC^1(\mathbb R;\mathbb R^n).
\end{equation}

To prove (\ref{contr1}),  let us  norm the space  $BC^1(\mathbb R;\mathbb R^n)$ with
\begin{equation}\label{beta}
\|v\|_{\sigma} = \|v\|_{BC(\mathbb R;\mathbb R^n)} + \sigma \|\partial_t v\|_{BC(\mathbb R;\mathbb R^n)},
\end{equation}
where a positive constant $\sigma$ will be specified later. Note that the norms given by (\ref{beta}) are equivalent for all $\sigma>0$.
 We, therefore, have  to prove that  there exist  constants $\sigma<1$ and $\gamma<1$ such that
$$
\|G_0 v\|_{BC(\mathbb R;\mathbb R^n)}+\sigma\left\|\frac{d}{dt} [G_0  v]\right\|_{BC(\mathbb R;\mathbb R^n)}
\le \gamma\left(\|v\|_{BC(\mathbb R;\mathbb R^n)} + \sigma\|v^\prime\|_{BC(\mathbb R;\mathbb R^n)}\right)
$$
for all $ v \in  BC^1(\mathbb R;\mathbb R^n)$. For given $v\in BC^1(\mathbb R;\mathbb R^n)$, we have
\begin{eqnarray*}
\frac{d}{dt} [G_0 v]_j(t) &=&  \partial_t c_j(x_j,1 - x_j,t)\sum\limits_{k=1}^nr_{jk}v_k(\omega_j(x_j, 1 - x_j, t)) \\
& &   +  c_j^1(x_j,1 - x_j,t) \sum\limits_{k=1}^nr_{jk}v_k^\prime(\omega_j(x_j, 1 - x_j, t)), \quad  j \le n.
\end{eqnarray*}
Here we used the equalities
\begin{eqnarray*}
 \partial_t \omega_j(\xi,x,t)= \exp\int_\xi^x \left[\frac{\partial_ta_j}{a_j^2}\right](\eta,\omega_j(\eta,x,t))d\eta
\end{eqnarray*}
and $c_j^1(\xi,x,t) = c_j(\xi,x,t)  \partial_t \omega_j(\xi,x,t).$
Define the operator $W \in \mathcal L(BC(\mathbb R; \mathbb R^n))$ by
 $$[W v]_j(t) =  \partial_t c_j(x_j,1 - x_j,t)\sum\limits_{k=1}^nr_{jk}v_k(\omega_j(x_j, 1 - x_j, t)),
 \quad j \le n. $$
Because of the assumption (\ref{G_i}) for $i=0$, we have $\left\|G_0\right\|_{\mathcal L(BC(\mathbb R; \mathbb R^n))}<1$. Hence, we can fix $\sigma<1$ such that  $\left\|G_0\right\|_{\mathcal L(BC(\mathbb R; \mathbb R^n))}+
\sigma\left\|W\right\|_{\mathcal L(BC(\mathbb R; \mathbb R^n))}<1$. Set
$$
\gamma=\max\left\{\left\|G_0\right\|_{\mathcal L(BC(\mathbb R; \mathbb R^n))}+
\sigma\left\|W\right\|_{\mathcal L(BC(\mathbb R; \mathbb R^n))}, \left\|G_1\right\|_{\mathcal L(BC(\mathbb R; \mathbb R^n))}
\right\}.
$$
We now  use the assumption (\ref{G_i}) for $i=1 $ and, therefore, conclude that the inequality $\left\|G_1\right\|_{\mathcal L(BC(\mathbb R; \mathbb R^n))}<1$ is fulfilled. This implies that $\gamma<1$. It follows that
$$
\begin{array}{rcl}
  \|G_0 v\|_{\sigma}
&\le&
\left\| G_0 v\right\|_{BC(\mathbb R;\mathbb R^n)}
 +\sigma \| W v\|_{BC(\mathbb R;\mathbb R^n)}
+ \sigma \left\|G_1 v^\prime\right\|_{BC(\mathbb R;\mathbb R^n)}\\
&\le& \gamma \left(\|v\|_{BC(\mathbb R;\mathbb R^n)} + 
\sigma\left\|v^\prime\right\|_{BC(\mathbb R;\mathbb R^n)}\right)=  \gamma\| v\|_{\sigma}.
\end{array}
$$
Furthermore,
$\|(I - G_0)^{-1}v\|_{\sigma} \le (1 - \gamma)^{-1}\|v\|_{\sigma}$
and, hence
\begin{equation}\label{ots441}
	\begin{split}
  \|(I - G_0)^{-1}v\|_{BC_t^1(\mathbb R;\mathbb R^n)} &\le \frac{1}{\sigma}\|(I - G_0)^{-1}v\|_{\sigma} \le \frac{1}{\sigma(1 - \gamma)}\|v\|_{\sigma} \\
 &\le\frac{1}{\sigma(1 - \gamma)}\|v\|_{BC_t^1(\mathbb R;\mathbb R^n)}. 
\end{split}
\end{equation}
Therefore, the system (\ref{simpl1})
can be expressed in the form
\begin{eqnarray} \label{z1}
z=(I-G_0)^{-1}\tilde h,
\end{eqnarray}
where $\tilde h(t)=\left(h_1(1,t),\dots,h_m(1,t),h_{m+1}(0,t),\dots,h_n(0,t)\right)$.
Substituting (\ref{z1}) into (\ref{simpl}), we obtain
\begin{equation}\label{uj}
	\begin{split}
u^*_j(x,t)&=\left[(I-C)^{-1}h\right]_j(x,t)  \\ 
&=\displaystyle c_j(x_j,x,t)\sum\limits_{k=1}^nr_{jk}\left[(I-G_0)^{-1}\tilde g\right]_k(\omega_j(x_j,x,t))+h_j(x,t)
\end{split}
\end{equation}
for all $ j \le n$. 
Combining  (\ref{uj}) with (\ref{ots441}) gives
\begin{equation}\label{I-C--1}
\|(I-C)^{-1}\|_{\mathcal L(BC^1_t(\Pi;\mathbb R^n))}\le 1+\frac{1}{\sigma(1 - \gamma)}\|C\|_{\mathcal L(BC^1_t(\Pi;\mathbb R^n))}.
\end{equation}

As demonstrated in \cite{KR}, the operators  $DC$ and $D^2$ exhibit smoothing properties. More precisely, they continuously map $BC(\Pi; \mathbb R^n)$
into $BC^1_t(\Pi; \mathbb R^n)$. This implies the existence of a positive constant
  $M_{11}$ such that for all
$u\in BC(\Pi;\mathbb R^n)$ we have
\begin{eqnarray} \label{ots3}
\left\|\partial_t\left[(DC + D^2)u\right]\right\|_{BC(\Pi;\mathbb R^n)} \le M_{11} \|u\|_{BC(\Pi;\mathbb R^n)}.
\end{eqnarray}

On the account of (\ref{oper2}),
\begin{equation}\label{h00}
\|u^*\|_{BC_t^1(\Pi;\mathbb R^n)} \le\displaystyle \|(I\!-\!C)^{-1}\|_{\mathcal L(BC^1_t(\Pi;\mathbb R^n))} \|(DC \!+\! D^2)u^*\!+\! (I \!+\! D)Ff\|_{BC_t^1(\Pi;\mathbb R^n)} .
\end{equation}
Combining  (\ref{h00}) with
(\ref{h5}),  (\ref{I-C--1}), and (\ref{ots3}), we obtain
$$
 \|u^*\|_{BC_t^1(\Pi;\mathbb R^n)}\le M_{12} \left(\| f\|_{BC^1(\Pi;\mathbb R^n)} +
  \| \partial^2_{xt}f\|_{BC(\Pi;\mathbb R^n)} \right),
$$
the constant $M_{12}$  being independent of  $f$.
Finally, from (\ref{eq:1}) we get the estimate
\begin{eqnarray*}
  \|\partial_x u^*\|_{BC(\Pi;\mathbb R^n)} &\le& \frac{1}{\Lambda_0}\left(\|f\|_{BC(\Pi;\mathbb R^n)}
 + \|b u^*\|_{BC(\Pi;\mathbb R^n)} + \|\partial_t u^*\|_{BC(\Pi;\mathbb R^n)}\right) \\
 & \le& M_{13} \left(\| f\|_{BC^1(\Pi;\mathbb R^n)} + \| \partial^2_{xt}f\|_{BC(\Pi;\mathbb R^n)} \right)
\end{eqnarray*}
for some  $M_{13}>0$ not depending on $f$.  
The estimate (\ref{est21}) now easily follows.

\subsection{Proof of Part ${\bf (ii)}$: $BC^2$-solutions}\label{BC2}

Differentiating   the system (\ref{eq:1})  with $u=u^*$ in a distributional sense in $t$ and the boundary conditions
 (\ref{eq:2})   pointwise, we get, respectively,
\begin{equation}\label{eq:112}
(\partial_t + a\partial_x)\partial_tu^* +
\left(b-a^{-1}\partial_ta\right)\partial_tu^*
+ \left(\partial_tb-a^{-1}\partial_ta \,b\right) u^*= \partial_t f - a^{-1}\partial_ta \,f
\end{equation}
and
\begin{equation} \label{eq:115}
\begin{array}{ll}
\displaystyle   \partial_tu^*_{j}(0,t)=R _j\partial_t u^*(\cdot,t),
 \quad 1\le j\le m,\; t\in \mathbb R,  \\ [2mm]
  \displaystyle \partial_tu^*_{j}(1,t)=R _j\partial_t u^*(\cdot,t),
 \quad m< j\le n,\; t\in \mathbb R.
\end{array}
\end{equation}
Using a similar argument as above, we state that the function $u^* \in BC^1(\Pi, \mathbb R^n)$ satisfies both (\ref{eq:112}) in distributional sense
and (\ref{eq:115}) pointwise
if and only if the function   $u^*$ 
satisfies the following system pointwise:
\begin{align}\label{cd2}
  \partial_tu^*_j(x,t)&=c_j^1(x_j,x,t) R_j \partial_tu^*(\cdot,\omega_j(x_j)) \nonumber\\ 
&\quad\displaystyle -\int_{x_j}^x d_j^1(\xi,x,t)
\biggl(\sum_{k\not=j} b_{jk}(\xi,\omega_j(\xi)) \partial_tu^*_k(\xi,\omega_j(\xi)) \nonumber\\ 
&\quad\displaystyle+  \biggl[\sum_{k=1}^n(\partial_t b_{jk} \!-\! a^{-1}_j\partial_t a_j \, b_{jk} )u^*_k \!-\! \partial_t f_j \!+\! a^{-1}_j\partial_t a_j \, f_j\biggr]
(\xi,\omega_j(\xi))\biggr)d\xi, j\!\le\! n.
\end{align}
 Rewrite  (\ref{cd2}) in the form
 \begin{eqnarray} \label{oper3}
  w = C_1w + D_1 w  + F_1(u^*,f),
 \end{eqnarray}
where $w=\partial_tu^*$ and  the operators
$C_1, D_1,F_1\in \mathcal L(BC(\Pi;\mathbb R^n))$ are defined, respectively,  by
\begin{eqnarray*}
& & [C_1w]_j(x,t) =
 c_j^1(x_j,x,t)R_j w(\cdot,\omega_j(x_j)),\\ [1mm]
& & [D_1w]_j(x,t) =
 -\int_{x_j}^{x}d_j^1(\xi,x,t)\sum_{k\neq j} b_{jk}(\xi,\omega_j(\xi))w_k(\xi,\omega_j(\xi)) d\xi,  \\
[1mm]
& & [F_1(u,f)]_j(x,t) = \\
& & -\int_{x_j}^x d_j^1(\xi,x,t)  \left[\sum_{k=1}^n(\partial_t b_{jk} - a^{-1}_j\partial_t a_j \, b_{jk} )u_k - \partial_t f_j + a^{-1}_j\partial_t a_j \, f_j\right]
(\xi,\omega_j(\xi))d\xi.
\end{eqnarray*}
 Iterating (\ref{oper3}), we obtain
 \begin{eqnarray} \label{oper4}  
 w = C_1w + (D_1 C_1 + D_1^2) w + (I+D_1) F_1(u^*,f).{\tiny } 
 \end{eqnarray}
  Similarly to the operators $DC$ and $D^2$, the
operators $D_1 C_1$ and $D_1^2$ are smoothing, mapping  continuously   $BC( \Pi; \mathbb R^n)$ into $BC^1_t(\Pi; \mathbb R^n).$
Moreover, the following  smoothing estimate is true:
\begin{equation}\label{smooth}
\left\|(D_1C_1 + D_1^2)w\right\|_{BC^1_t(\Pi;\mathbb R^n)} \le M_{14} \|w\|_{BC(\Pi;\mathbb R^n)}
\end{equation}
for some $ M_{14} >0$ not depending on $w\in BC(\Pi;\mathbb R^n)$.

Again, similarly to the bijectivity of the operator $I-C$, we prove that $I - C_1$ is a bijective operator   from $BC^1_t( \Pi; \mathbb R^n)$ to
itself.
In other words, we have to show that the system
$$
w_j(x,t)=c_j^1(x_j,x,t)R_jw(\cdot,\omega_j(x_j,x,t))+h_j(x,t), \quad j\le n,
$$
is uniquely solvable in $BC^1_t (\Pi;\mathbb R^n)$ for each $h\in BC^1_t (\Pi;\mathbb R^n)$ or, the same,
  that the operator
$
I-G_1
$
 is a bijective operator from    $BC^1(\mathbb R;\mathbb R^n)$ to itself,
where the operator $G_1$ is given by (\ref{eq:G}).
To this end, we use the assumptions 
that
 $\|G_1\|_{\mathcal L(BC(\mathbb R; \mathbb R^n))}<1$ and  $\|G_2\|_{\mathcal L(BC(\mathbb R; \mathbb R^n))}<1$.

Similarly to (\ref{I-C--1}), the inverse to $I-C_1$ fulfills the bound
\begin{equation}\label{i1}
\|(I-C_1)^{-1}\|_{\mathcal L(BC^1_t(\Pi;\mathbb R^n))}\le 1+M_{15}\|C_1\|_{\mathcal L(BC^1_t(\Pi;\mathbb R^n))}
\end{equation}
for some $M_{15}>0$.
Further, due to (\ref{oper4}), it holds
\begin{eqnarray}\label{h000}
 \|w\|_{BC_t^1(\Pi;\mathbb R^n)} 
& \le& \displaystyle \|(I-C_1)^{-1}\|_{\mathcal L(BC^1_t(\Pi;\mathbb R^n))} \|(D_1C_1 + D_1^2)w \nonumber\\
 & &  + (I + D_1)F_1(u^*,f)\|_{BC_t^1(\Pi;\mathbb R^n)} .
\end{eqnarray}
Combining  (\ref{h000}) with
(\ref{est21}),  (\ref{smooth}), and (\ref{i1}), we conclude that there exists $M_{16}>0$ such that
$$
\|w\|_{BC^1_t(\Pi;\mathbb R^n)} \le 
M_{16}\left(\| f\|_{BC^1(\Pi;\mathbb R^n)}+ \|\partial_{xt}^2 f\|_{BC(\Pi;\mathbb R^n} +
\|\partial_t^2 f\|_{BC\left(\Pi;\mathbb R^n\right)}\right).
$$
Hence,
\begin{equation}\label{h55}
\|\partial_t^2 u^*\|_{BC(\Pi;\mathbb R^n)} \le M_{16}\left(\| f\|_{BC^1(\Pi;\mathbb R^n)}\!+\! \|\partial_{xt}^2 f\|_{BC\left(\Pi;\mathbb R^n\right)} \!+\!
\|\partial_t^2 f\|_{BC\left(\Pi;\mathbb R^n\right)}\right).
\end{equation}
From (\ref{eq:112}) we now easily get the bound 
$$
\|\partial^2_{xt} u^*\|_{BC(\Pi;\mathbb R^n)}  \le M_{17} \left(\| f\|_{BC^1(\Pi;\mathbb R^n)}+ \|\partial_{xt}^2 f\|_{BC\left(\Pi;\mathbb R^n\right)} +
\|\partial_t^2 f\|_{BC\left(\Pi;\mathbb R^n\right)}\right).
$$
Differentiating  the system (\ref{eq:1}) in $x$, we additionally obtain an estimate for $\partial_x^2 u^*$, namely,
$$
\|\partial^2_{x} u^*\|_{BC(\Pi;\mathbb R^n)}  \le M_{18} \left(\| f\|_{BC^1(\Pi;\mathbb R^n)}+ \|\partial_{xt}^2 f\|_{BC\left(\Pi;\mathbb R^n\right)} +
\|\partial_t^2 f\|_{BC\left(\Pi;\mathbb R^n\right)}\right),
$$
where the constant $M_{18} $ is independent of $f$.
 The estimate (\ref{est22}) now easily follows. The proof of Part ${\bf (ii)}$ of the theorem is complete.
 
 \subsection{Proof of Part ${\bf (iii)}$:  Almost periodic solutions}  Now we assume that the coefficients of the initial problem are almost periodic 
 in $t$ and  prove that the  solution $u^*$ constructed above
belongs to~$AP(\Pi;\mathbb R^n)$. 
 
 Fix $\mu > 0$ and let $h$ be a common $\mu$-almost period of the functions $a_j$, $b_{jk}$, and $f_j$
 as well as
their derivatives in $x$ and $t$.
Then the functions $\tilde a_j(x,t) = a_j(x,t + h) - a_j(x,t)$ and $\tilde b_{jk}(x,t) = b_{jk}(x,t + h) - b_{jk}(x,t)$ satisfy
the inequalities
\begin{eqnarray} \label{tilde}
 \|\tilde a_j\|_{BC^1(\Pi)} \le \mu\quad \mbox{ and }\quad \|\tilde b_{jk}\|_{BC^1(\Pi)} \le \mu.
\end{eqnarray}
The desired statement will be proved if we show that $h$ is an almost period of the function~$u^*(x,t).$

While  the function $u^*(x,t)$ is a unique bounded  classical solution
to the problem (\ref{eq:1}), (\ref{eq:2}), the function 
  $u^*(x,t+h)$ is a unique bounded  classical solution to the  system
$$
 \partial_t u + a(x,t+h)\partial_x u  + b(x,t+h)u = f(x,t+h)
$$
subjected to the boundary conditions (\ref{eq:2}).
Then the difference $z(x,t) = u^*(x,t) - u^*(x,t+h)$ satisfies the system
$$
\partial_t z + a(x,t) \partial_x z + b(x,t) z = g(x,t)
 $$
 and the boundary conditions (\ref{eq:2}), where
 \begin{align*}
   g(x,t)
  &= - \left(b(x,t) -  b(x,t+h)\right)u^*(x,t+h) \\ 
 & \quad - \left(a(x,t) -   a(x,t+h)\right)\partial_x u^*(x,t+h)
 +  f(x,t) - f(x,t+h).
\end{align*}
 The function $g(x,t)$ is $C^1$-smooth in $x$ and $t$ and, due to (\ref{est21}) and (\ref{tilde}),
  satisfies the bound
  \begin{equation}\label{th-6h}
  	\begin{split}
 \|g\|_{BC\left(\mathbb R;L^2((0,1);\mathbb R^n)\right)}&\le  \|g\|_{BC(\Pi;\mathbb R^n)} \\ 
 & \le \mu\left[L_{1}\left(\| f\|_{BC^1(\Pi;\mathbb R^n)} + \| \partial^2_{xt}f\|_{BC(\Pi;\mathbb R^n)} \right)
 +1\right].
\end{split}
\end{equation}
Now, by Theorem \ref{u_*},   the function
$y: \mathbb{R} \to L^2((0,1);\mathbb R^n)$ defined by $\left[y(t)\right](x)=z(x,t)$
 is given by the formula 
\begin{eqnarray*}  
y(t) = \int_{-\infty}^{t} U(t,s) g(s) d s,
\end{eqnarray*}
where $\left[g(t)\right](x)=g(x,t)$.
Moreover, on account of (\ref{th-6h}),
it satisfies  to prove the following  estimate:
\begin{equation}	\label{th-33wh}
  \begin{array} {rcl}
 	 \| y(t)\|_{L^2((0,1);\mathbb R^n)} &\le &\displaystyle\frac{M}{\alpha}\|g\|_{BC\left(\mathbb R;L^2((0,1);\mathbb R^n)\right)}  \nonumber\\[3mm]
 	& \le &\displaystyle\frac{\mu M}{\alpha}\left[L_{1}\left(\| f\|_{BC^1(\Pi;\mathbb R^n)} + \| \partial^2_{xt}f\|_{BC(\Pi;\mathbb R^n)} \right)+1\right].
 \end{array}
\end{equation}
 It follows that  the function $v^*$ 
 is almost periodic in $t$ with values in $L^2((0,1);\mathbb R^n)$. Hence, the function $u^*(x,t)=[v^*(t)](x)$
is almost periodic in $t$ in the mean in $x$, see \cite[Theorem 2.15]{Cord}.
Moreover, since  $u^*\in BC^2(\Pi, \mathbb R^n),$  the derivative
 $\partial_t u^*(t,x)$ is uniformly continuous in $t$ in the mean in $x$.
On account of   \cite[Theorem 2.13]{Cord},   we conclude that $\partial_t u^*(t,x)$ is almost periodic in $t$ in the mean in $x$.
 The almost periodicity of $\partial_x u^*(x,t)$ now easily follows from the equation (\ref{eq:1}).
 Consequently,  the function $u^*(x,t)$ is almost periodic in $t$ uniformly in $x$, as desired.
 
 \section{Periodic solutions: proof of Theorem \ref{periodic}}\label{per}
\setcounter{equation}{0}

Let the conditions of Theorem~\ref{lin-smooth} ${\bf (i)}$ are fulfilled.
	Note that those conditions  fall into the scope of   Theorem~\ref{u_*}. Accordingly to Theorem~3.3, the problem 
	 (\ref{unperturb})
	has a unique bounded classical solution~$v^*$,  given by the formula (\ref{boundsol}). As  the coefficients $a_j$ and $b_{jk}$ are $T$-periodic in $t$, we have $\mathcal A(t+T)=\mathcal A(t)$ and, hence
	$U(t+T,s+T)=U(t,s)$ for all $t\ge s$. This and the solution formula 
	(\ref{boundsol})
	imply that $v^*(t+T)=v^*(t)$ for all $t\in\mathbb R$. Moreover, as $v^*$ is the classical solution, 
		then it is continuous in $t\in\mathbb R$ with values in $H^1((0,1);\mathbb R^n)$. On account of the periodicity of $v^*$,
		the solution 	$v^*$ belongs to $BC(\mathbb R;H^1((0,1);\mathbb R^n))$ and, hence, to $BC(\mathbb R;C([0,1];\mathbb R^n))$.	
	This implies  that the function $u^*(x,t)=[v^*(t)](x)$
belongs to $C_{per}(\Pi;\mathbb R^n)$. Then, due to the proof of Theorem~\ref{lin-smooth}
(see Section \ref{BC0}), $u^*$ satisfies the equation (\ref{rep}) or, the same, its operator 
form  (\ref{oper}) pointwise.

To finish the proof of Claim ${\bf (i)}$, it remains to prove the a priori estimate (\ref{est211}). We show
that the map $f\in  C_{per}(\Pi;\mathbb R^n)\to u^* \in  C_{per}(\Pi;\mathbb R^n)$ is continuous. To this end, if we prove 
that the operator $I- C -D$ is Fredholm of index zero from $C_{per}(\Pi;\mathbb R^n)$ to itself and that it
is injective from $C_{per}(\Pi;\mathbb R^n)$ to itself.
This  will then imply the bijectivity of $I- C -D$ and the boundedness of the inverse 
$(I- C -D)^{-1}$, as desired.

Despite the Fredholmness of $I- C -D$  follows from \cite[Theorem 1.2]{KK} and \cite[Theorem 1.2]{KR},
 we
give a short proof of the Fredholmness
property,  using the notion of parametrix: 
First note that the operator $I-C$ is bijective  from $C_{per}(\Pi;\mathbb R^n)$ to itself. Indeed, similarly to the proof of 
$BC^1$-smoothness in Section \ref{BC1}, the bijectivity of $I-C$
 is equivalent to the
 unique solvability 
 of  the  system (\ref{simpl}) with respect to $u$ in $C_{per}(\Pi;\mathbb R^n)$ or, the same, to the
 unique solvability  
 of  the  system (\ref{simpl1}) with respect to $z$ in $C_{per}(\mathbb R;\mathbb R^n)$. The last statement is, in its turn, equivalent to 
 the bijectivity of $I-G_0$ from $C_{per}(\mathbb R;\mathbb R^n)$ to itself. But this is true due to the assumption 
(\ref{G_i})  for $i=0$ and the Banach fixed-point theorem. 

Taking into account the bijectivity of $I-C$,  the operator $I- C -D$ is Fredholm of index zero 
from $C_{per}(\Pi;\mathbb R^n)$ to itself
if and only if the operator $I-(I-C)^{-1}D$ is Fredholm operator of index zero from $C_{per}(\Pi;\mathbb R^n)$ to itself.
Further we use the Fredholmness criterion given by \cite[Proposition 1, pp. 298--299]{Zeidler}. It states that, if there exist linear continuous operators $P_l, P_r: C_{per}(\Pi;\mathbb R^n)\to C_{per}(\Pi;\mathbb R^n)$ (called parametrices) and compact operators
$Q_l,Q_r: C_{per}(\Pi;\mathbb R^n)\to C_{per}(\Pi;\mathbb R^n)$ such that
$$
P_l\left[I-(I-C)^{-1}D\right]=I+Q_l\quad\mbox{and}\quad \left[I-(I-C)^{-1}D\right]P_r=I+Q_r,
$$ 
then the operator $I-(I-C)^{-1}D$ is Fredholm of index zero. Take
$$
P_l=P_r=I+(I-C)^{-1}D
$$
and calculate
$$
P_l\left[I-(I-C)^{-1}D\right]=\left[I-(I-C)^{-1}D\right]P_r=I-(I-C)^{-1}D(I-C)^{-1}D.
$$
Accordingly to  \cite[Proposition 1, pp. 298--299]{Zeidler}, it remains
to prove the compactness of the operator $(I-C)^{-1}D(I-C)^{-1}D$ or, on the account of the 
boundedness of  $(I-C)^{-1}$, the compactness of the operator $D(I-C)^{-1}D: C_{per}(\Pi;\mathbb R^n)\to C_{per}(\Pi;\mathbb R^n)$. Indeed,
$$
D(I-C)^{-1}D= D\left[I+C(I-C)^{-1}\right]D=D^2+DC(I-C)^{-1}D.
$$
Because of the boundedness of the operator $(I-C)^{-1}D$, we are done if we prove the compactness of the operators
$D^2$ and $DC$. But this follows immediately from  \cite[Equation (4.2)]{KR}
and  the Arzela-Ascoli theorem. The proof of the Fredholmness of  $I- C -D$ is, therefore, complete.

To prove the injectivitiy of $I- C -D$, 
suppose that the equation (\ref{oper}) has two continuous periodic solutions, say $u^1$ and $u^2$. Then the difference
$u^1-u^2$ satisfies the equation (\ref{oper}) with $f=0$ and, consequently, the equation (\ref{oper2}) with $f=0$.
Since the operators
$D^2$ and $DC$ map continuously $C_{per}(\Pi;\mathbb R^n)$ into $C_{per}^1(\Pi;\mathbb R^n)$
and the operator $(I-C)^{-1}$ maps $C_{per}(\Pi;\mathbb R^n)\cap BC_t^1(\Pi;\mathbb R^n)$ into itself
(see Section~\ref{BC1}),  we
conclude that
$u^1-u^2\in C_{per}(\Pi;\mathbb R^n)\cap BC_{t}^1(\Pi;\mathbb R^n)$.  Using additionally that $u_1-u_2$ is a distributional solution to (\ref{eq:1u}), (\ref{eq:2}),
we get that $u^1-u^2\in C_{per}^1(\Pi;\mathbb R^n)$. This means that 
$u^1-u^2$ is a classical solution to
(\ref{eq:1u}), (\ref{eq:2}).

For an arbitrary fixed $s\in\mathbb R$, the function $u^1(\cdot,s)-u^2(\cdot,s)$ belongs to  $C([0,1];\mathbb R^n)$ and, hence, to $L^2((0,1);\mathbb R^n)$. Then,
by Definition \ref{L2}, the function $u^1-u^2$ is also the $L^2$-generalized solution to the problem 
(\ref{eq:1u}), (\ref{eq:2}), (\ref{eq:in}). Due to the exponential decay
 estimate (\ref{est1}),  
 $u^1-u^2$ decays to zero as $t\to\infty$.  As $u^1-u^2$ is $T$-periodic, this yields that
  $u^1(x,t)= u^2(x,t)= 0$ for all $x\in[0,1]$ and $t\in \mathbb R$. The injectivity is therewith proved.
  
 As it follows from the above,
  $
   u^*=(I- C -D)^{-1}Ff
   $
  and the following a priori estimate is fulfilled:
  \begin{equation}\label{h555}
  \|u^*\|_{BC(\Pi;\mathbb R^n)} \le L_5\| f \|_{BC(\Pi;\mathbb R^n)}
\end{equation}
  for some $L_5>0$ not depending on $ f \in BC(\Pi;\mathbb R^n)$. Moreover,  combining  (\ref{h00}) with  (\ref{I-C--1}),  (\ref{ots3}), and (\ref{h555}),   we obtain the a priori estimate
 \begin{equation}\label{h666}
 \|u^*\|_{BC^1_t(\Pi;\mathbb R^n)} \le L_6\| f \|_{BC^1_t(\Pi;\mathbb R^n)}.
\end{equation}
 The final estimate (\ref{est211}) now straightforwardly follows from (\ref{eq:1}).
This finishes the proof of  Claim  ${\bf (i)}$ of Theorem \ref{periodic}.

To prove Claim ${\bf (ii)}$ of Theorem \ref{periodic}, we 
first use Claim ${\bf (ii)}$ of Theorem \ref{lin-smooth}
and conclude that $u^*\in C^2_{per}(\Pi;\mathbb R^n)$. To prove
the estimate (\ref{L6}), we
follow a similar argument as in the proof of the estimate (\ref{est22})
in Section \ref{BC2}, with the following changes.
Combining (\ref{h000}) with (\ref{smooth}), (\ref{i1}), (\ref{h666}), we get, instead of (\ref{h55}), the estimate 
$$\|\partial_t^2u^*\|_{BC(\Pi;\mathbb R^n)} \le L_7\| f \|_{BC^2_t(\Pi;\mathbb R^n)}.$$
Then from (\ref{eq:112}) we immediately get a similar bound for $\partial^2_{xt}u^*$.
Differentiating  the system (\ref{eq:1}) in $x$, we get also the bound for $\partial_x^2 u^*$, namely
	$$
	\|\partial^2_{x} u^*\|_{BC(\Pi;\mathbb R^n)}  \le L_{8} \left(\|\partial_x f\|_{BC(\Pi;\mathbb R^n)}+
	\|\partial_t^2 f\|_{BC\left(\Pi;\mathbb R^n\right)}\right).
	$$
 The desired  bound (\ref{L6}) now easily follows. This completes the proof of  Theorem~\ref{periodic}.

\section*{Acknowledgments}
The authors were supported by the Volkswagen Foundation grant 
	``From Modelling and Analysis to Approximation". V. Tkachenko was also supported by the
	Volkswagen Foundation grant ``Dynamic Phenomena in Elasticity Problems" and
	by the DFG project SFB/TRR 109 ``Discretization in Geometry and Dynamics", project  N195170736.

\end{document}